\documentclass[leqno,12pt]{amsart}
\usepackage[margin=1.1in]{geometry}

\usepackage{amssymb,mathdots,epsfig}
\usepackage{enumitem}
\usepackage[T1]{fontenc}
\usepackage{fullpage}
\usepackage{url} 
\usepackage[colorlinks,unicode,pdfencoding=auto]{hyperref} % better urls in bibliography
\hypersetup{linkcolor=blue, citecolor=blue}

\usepackage{graphicx}

\theoremstyle{plain}
\newtheorem{theorem}{Theorem}
\newtheorem{corollary}{Corollary}

\newtheorem{lemma}{Lemma}

\theoremstyle{definition}
\newtheorem{definition}{Definition}
\newtheorem{example}{Example}

\theoremstyle{remark}

\numberwithin{equation}{section}

\usepackage{array,longtable}
\newcolumntype{C}{>{$}c<{$}}

\newcommand{\Z}{\mathbb{Z}}
\newcommand{\N}{\mathbb{N}}

\newcommand{\<}{\langle}
\renewcommand{\>}{\rangle}

\newcommand{\lf}{\left}
\newcommand{\rh}{\right}

\renewcommand{\(}{\begin{equation}}
\renewcommand{\)}{\end{equation}}

\newcommand{\midtilde}{\raisebox{-0.25\baselineskip}{\textasciitilde}}

\makeatletter
\newcommand{\leqnomode}{\tagsleft@true}
\newcommand{\reqnomode}{\tagsleft@false}
\makeatother

\setbox0=\hbox{$+$}
\newdimen\plusheight
\plusheight=\ht0
\def\+{\;\lower\plusheight\hbox{$+$}\;}

\setbox0=\hbox{$-$}
\newdimen\minusheight
\minusheight=\ht0
\def\-{\;\lower\minusheight\hbox{$-$}\;}

\setbox0=\hbox{$\cdots$}
\newdimen\cdotsheight
\cdotsheight=\plusheight
\def\cds{\lower\cdotsheight\hbox{$\cdots$}}

\makeatletter
\def\section{%
    \@startsection{section}{1}%
    \z@{3\linespacing\@plus\linespacing}{.5\linespacing}%
    {\normalfont\large\bfseries}%
}
\makeatother

\makeatletter
\def\@seccntformat#1{%
  \protect\textup{\protect\@secnumfont
    \csname the#1\endcsname
\space\space
  }%
}
\makeatother

\makeatletter
\def\subsection{\@startsection{subsection}{2}%
  \z@{1.5\linespacing\@plus.5\linespacing}{-.5em}%
  {\normalfont\bfseries}}
\makeatother

\newcommand{\xa}{\alpha}
\newcommand{\xb}{\beta}
\newcommand{\xc}{\chi}
\newcommand{\xd}{\delta} % renew
\newcommand{\xe}{\epsilon}

\newcommand{\xg}{\gamma}

\newcommand{\xk}{\kappa}

\newcommand{\xr}{\rho}
\newcommand{\xs}{\sigma}

\newcommand{\xS}{\Sigma}

\usepackage{fancyhdr}
\newcommand{\tfr}[2]{\tfrac{#1}{#2}}
\newcommand{\tf}[2]{\tfrac{#1}{#2}}
\newcommand{\fr}[2]{\frac{#1}{#2}}

\makeatletter
\renewcommand{\pmod}[1]{\allowbreak\mkern6mu({\operator@font mod}\,\,#1)}
\makeatother

\renewcommand{\mod}[1]{\,\,(\mathrm{mod}\,#1)}

\renewcommand{\l}{\ell}

\renewcommand{\epsilon}{\varepsilon}

\newcommand{\tops}[2]{\texorpdfstring{#1}{#2}}
\renewcommand{\.}{\mspace{1.5mu}} % half of a \,
\newcommand{\?}{\mspace{-1.5mu}} % half of a \!

\DeclareMathOperator{\sgn}{sgn}

\usepackage{doi}

\begin{document}

%%%%%%%%%%%%%%%%%%%%%%%%%%%%%%%%%%%%%%%%%%
% TITLE MATTER

\title[The $p$-Dissection of a Product]
       {The $p$-Dissection of a Product of Quintuple Products}

\author{Taylor Daniels, Timothy Huber, James McLaughlin, and Dongxi Ye}

\address{Mathematics Deptartment, Purdue University, West Lafayette, IN 47907}
\email{daniel84@purdue.edu}

\address{
School of Mathematical and Statistical Sciences, University of Texas Rio Grande Valley, Edinburg, TX 78539}
\email{timothy.huber@utrgv.edu}

\address{Mathematics Department\\
West Chester University, West Chester, PA 19383}
\email{jmclaughlin2@wcupa.edu}

\address{
School of AI and Liberal Arts \\ Beijing Normal-Hong Kong Baptist University, Zhuhai 519082, Guangdong,
People's Republic of China}
\email{dongxiye@bnbu.edu.cn}

\thanks{Dongxi Ye was supported by the Guangdong Basic and Applied Basic Research Foundation (Grant No. 2024A1515030222).}


% \date{\today}

\begin{abstract}
    Let $p \equiv 1 \pmod{4}$ be prime, let $m$ and $n$ be integers such that $p=m^2+n^2$, and let $b$ be a positive integer. Let $Q(z,q) = (z,q/z,q;q)_{\infty}(qz^2,q/z^2;q^2)_{\infty}$ denote the product appearing in the quintuple product identity. We derive explicit formulae for the $p$-dissection of $Q(q^{bm},q^p)Q(q^{bn},q^p)$, and determine sign patterns in \mbox{length-$p$} arithmetic progressions of the Taylor series coefficients of the associated quotient $Q(q^{bm},q^{p})Q(q^{bn},q^p)/(q^p;q^p)_{\infty}^2$. Some combinatorial applications of the $p$-dissection formulae are also given.
\end{abstract}

\allowdisplaybreaks

\maketitle

%%%%%%%%%%%%%%%%%%%%%%%%%%%%%%%%%%%%%%%%%%
% /TITLE MATTER

%%%%%%%%%%%%%%%%%%%%%%%%%%%%%%%%%%%%%%%%%%
\section{Introduction}
%%%%%%%%%%%%%%%%%%%%%%%%%%%%%%%%%%%%%%%%%%

Using the standard $q$-Pochhammer symbols $(a_1,  \dots, a_j; q)_{\infty} = (a_1;q)_{\infty} \cdots (a_j;q)_{\infty}$, let
    \(
        \label{qtpzqdef}
            Q(z,q) := (z,q/z,q;q)_{\infty}(qz^2,q/z^2;q^2)_{\infty}
    \)
denote the \emph{quintuple product}.
In a recent paper \cite{DMcL25}, two of the present authors prove the following result: \emph{If one defines the quantities $b_{t}$ via
    \[
        B(q):= Q(q^2,q^{13})Q(q^5,q^{13})Q(q^6,q^{13}) =: \sum_{t=0}^{\infty} b_t q^t,
    \]
then $b_{13t+3} = b_{13t+9} = b_{13t+11} = 0$, for all $t \geq 0$.} This, along with numerical experiments, naturally motivates questions about vanishing coefficients in general expressions involving quintuple products $Q(q^{bn},q^{p})$, in particular when $p$ is a prime with \mbox{$p \equiv 1 \mod{4}$}. Moreover, it is of general interest to determine possible sign patterns in the Taylor series coefficients of such expressions.  

In this paper we primarily consider coefficients arising in the following manner: Fixing a prime \mbox{$p \equiv 1 \mod{4}$}, let $m$ and $n$ be integers such that $p = m^{2} + n^{2}$, and let $b > 0$ be some integer. For this $b$, $m$, $n$, and $p$ then, we define the sequence $(a_{t})_{t \in \Z}$ via the series expansion
    \(
    \label{eq:atDefin}
        Q(q^{bm},q^{p})Q(q^{bn},q^{p}) = 
        \sum_{t=-\infty}^{\infty} a_{t}q^{t}.
    \)
We are particularly interested in vanishings patterns of the $a_{t}$ within ``arithmetic progressions modulo $p$'', i.e., sequences of the form $(a_{pt+r})_{t\in\Z}$, where $0 \leq r < p$. Our primary result concerning these vanishings is given in the following theorem.

\begin{theorem}
    \label{thm:MainVanishing}
    Fix a prime $p \equiv 1 \mod{4}$, write $p = m^2 + n^2$, let $b > 0$, and define \mbox{$(a_{t})_{t \in \Z}$} as in \eqref{eq:atDefin}. In addition, let $w$ satisfy
        \[
            w \equiv \bar{2}(m+n) \mod{p}.
        \]
    If $p \equiv 1 \mod{12}$, then
        \[
            a_{pt+bw} = a_{pt + b(w-3b)} = 0 \qquad (t \in \Z),
        \]
    and if $p \equiv 5 \mod{12}$, then
        \[
            a_{pt+bw(1-3b\bar{m})} = a_{pt+bw(1-3b\bar{n})} = 0 \qquad (t \in \Z).
        \]
\end{theorem}

Although a direct proof of Theorem \ref{thm:MainVanishing} is possible by constructing sign-changing involutions of the series \eqref{eq:atDefin} (dependent on the residue $p$ modulo 12), we establish Theorem \ref{thm:MainVanishing} as a corollary to some general \emph{$p$-dissection} formulae for said series; we recall that a $p$-dissection of a series $A(q) = \sum_{t} a_t q^{t}$ is a formula 
    \[
        A(q) = A_{0}(q^{p}) + q A_{1}(q^{p}) + \cdots + q^{p-1} A_{p-1}(q^{p}),
    \]
where each $A_{i}(q)$ is some series $\sum_{t} a_{t}^{(i)} q^{t}$. Our dissections of \eqref{eq:atDefin} are built upon a general result of Liu and Yang \cite[Thm.~2]{LY09} concerning products
    \[
        \<z_{1};q\>_{\infty}\<z_{2};q\>_{\infty},
        \qquad\text{where}\qquad 
        \<z;q\>_{\infty} := (z,q/z,q;q)_{\infty}.
    \]

Just as the cases $p \equiv 1 \mod{12}$ and $p \equiv 5 \mod{12}$ were distinguished in Theorem \ref{thm:MainVanishing}, our $p$-dissections of  \eqref{eq:atDefin} also treat these cases separately. For illustration, a dissection formula for the case $p \equiv 1 \mod{12}$ is given in the following theorem; the full result is stated and proved as Theorem \ref{Qdissect1} in section \ref{sec:1(mod12)}.

\begin{theorem} \label{t1}
    Let $p \equiv 1 \pmod{12}$ with $p=m^2+n^2$ and $3 \mid n$, and let $b$ be a positive integer. If $m \equiv 1 \pmod{3}$, then 
        \[
        \begin{aligned}
            & Q(q^{bm},q^p)Q(q^{bn},q^p) \\
            &\qquad = \sum_{k=0}^{p-1} q^{3bkm+\fr{k(3k-1)p}{2}}
            Q\?\big({q^{bp+kmp + \fr{(p-m-n)p}{6}},q^{p^2}}\big) Q\?\big({q^{-knp + \fr{(p-m+n)p}{6}},q^{p^2}}\big).
        \end{aligned}
        \]
\end{theorem}

When $p \equiv 5 \pmod{12}$, our dissections involve a product appearing in Winquist’s identity. One such formula is stated below, and the full result is given in Theorem \ref{thm:QuintDiss(5mod12)}.

\begin{theorem}
    \label{t2}
    Let $p \equiv 5 \mod{12}$ with $p = m^2+n^2$, let $b$ be a positive integer, and let
        \[
            W(a,b,q) = \<a,b,ab,\tf{a}{b};q\>_{\infty}(q;q)_{\infty}^{-2}.
        \] 
    Defining $\mu$ via the relation $3m\mu \equiv 1 \mod{p}$, let
        \[
            s := (3m\mu-1)/p.
        \]
    If $m \equiv -n \mod{3}$, and one sets
        \[
            \xa_{k} := 3kmp + 3bp - \fr{(m+n)p}{2} + \fr{3p^2}{2}, \qquad 
            \xb_{k} := 3knp + \fr{(m-n)p}{2} + \fr{3p^2}{2},
        \]  
        \(
            \xc_{k} := 3bkm + p\Big[\tfr{k(3k-1)}{2} + (1-ns)(kn+\tfr{(p+m-n(1+ps))}{6})\Big],
        \)
    then one has
        \(
        \label{eq:WinqDiss(m=-n(3))2}
            Q(q^{bm},q^{p})Q(q^{bn},q^{p}) = \sum_{k=0}^{p-1} q^{\xc_{k}}W(-q^{\xa_{k}/3},-q^{(\xb_{k}-nsp^2)/3},q^{p^2}).
        \)
\end{theorem}

The $p$-dissections of Theorems \ref{t1} and \ref{t2} (and more generally Theorems \ref{Qdissect1} and \ref{thm:QuintDiss(5mod12)}) allow us to derive explicit  sign patterns for the series coefficients of the quotients
    \( \label{w1} 
        \frac{Q(q^{bm}, q^{p}) Q(q^{bn},q^{p})}{(q^{p}; q^{p})_{\infty}^{2}}.
    \)
Namely, Corollaries \ref{Qdissectanvanceqsign1} and \ref{Qdissectanvanceqsign2} show that, after finitely many initial zero terms, the signs of the series coefficients of \eqref{w1} follow explicit, predictable patterns; we demonstrate some examples following said corollaries. Finally, the form of the dissection components allow us to show that certain series coefficients vanish modulo $2$.

To illustrate some applications of our results, we mention a few examples here. Example \ref{p13vancex} shows that, for $p=13$ and $a_t$ defined via $Q(q^{10},q^{13})Q(q^{15},q^{13}) = \sum_{t} a_t q^t$, one has $a_{13t+6}=a_{13t+9}=0$ for all $t$. 
Example \ref{p13signpat} shows that (apart from some initial zero coefficients) for 
$$
\frac{Q(q^{10},q^{13})Q(q^{15},q^{13})}{(q^{13};q^{13})_{\infty}^{2}}=\sum_{t=0}^{\infty}b_{t}q^{t},
$$
that the signs of the sequence of coefficients $b_t$ is periodic modulo 13, and states this sign pattern explicitly. 
Examples \ref{p=17vanex} and~\ref{p17signpat} exhibit similar results for series coefficients of the products $Q(q^2,q^{17})Q(q^8,q^{17})$ and $Q(q^2,q^{17})Q(q^8,q^{17})/(q^{17};q^{17})_{\infty}^{2}$, respectively.
Example  \ref{p=17mod2ex} shows for $Q(q^2,q^{17})Q(q^8,q^{17})$ that $a_{17t+5} \equiv a_{17t+10} \equiv 0 \pmod{2}$. 
Example \ref{p13ex} explicitly computes the~13-dissection of $Q(q^2,q^{13})Q(q^3,q^{13})$, and hence derives several partition theoretic results from the 13-dissection of $Q(q^2,q^{13})Q(q^3,q^{13})/(q^{13};q^{13})_{\infty}^2$. 
Example \ref{p17ex}  similarly demonstrates the 17-dissection of $Q(q^2,q^{17})Q(q^8,q^{17})$, and likewise derives some partition theoretic consequences.
\vspace{\baselineskip}

\noindent\textbf{The structure of this paper.} First, section~\ref{sec:prelims} records notation and preliminary lemmas, including the including the dissection identities for the triple product. Section \ref{sec:1(mod12)} proves the main $p$-dissection formulae for the case $p \equiv 1 \pmod{12}$ and derives vanishing and sign consequences. Section \ref{sec:5(mod12)} treats 
the case $p \equiv 5 \pmod{12}$ via Winquist’s identity. Section \ref{sec:combinatorial} gives combinatorial interpretations for some cases, and Section \ref{sec:conclusions} gives concluding remarks and some open problems.
\vspace{\baselineskip}

\noindent\textbf{A computational aid.} A number of our derivations involve cumbersome, tedious symbolic manipulations. To help keep these cumbersome formulae free of ``typos'', and to aid the curious reader, a \emph{Mathematica} notebook further detailing and implementing our computations is freely, publicly available on the first author's website\footnote{\texttt{math.purdue.edu/{\midtilde}daniel84/}}, and upon request. We remark, however, that neither use of- nor proficiency with- \emph{Mathematica} is essential in any part of this paper's results.

%%%%%%%%%%%%%%%%%%%%%%%%%%%%%%%%%%%%%%%%%%
\section{Preliminaries} \label{sec:prelims}
%%%%%%%%%%%%%%%%%%%%%%%%%%%%%%%%%%%%%%%%%%

Throughout this paper $p$ denotes a prime satisfying $p \equiv 1 \mod{4}$, and $m$ and $n$ are integers satisfying $p = m^{2} + n^{2}$, unless stated otherwise. Overbarred quantities, such as $\bar{m}$, indicate multiplicative inverses modulo $p$, and expressions $\fr{a}{b} \mod{p}$ with $(b,p)=1$ are understood as $a\bar{b} \mod{p}$. Except for the complex quantities $q$ and $z$, lowercase Latin letters such as $a$, $k$, $t$, etc.~always denote integers. Statements that hold ``for all $t > t_{0}$'' (or similar) are understood to mean ``there is some constant $t_{0}$ for which [the statement] holds whenever $t > t_{0}$''.

Expressions $\<z;q\>_{\infty}$ indicate the \emph{triple product}
    \[
        \<z;q\>_{\infty} = (z,q/z,q;q)_{\infty}, \quad\text{and we let}\quad 
        \<z_{1},\ldots,z_{k};q\>_{\infty} := \<z_{1};q\>_{\infty}\cdots\<z_{k};q\>_{\infty}.
    \]
We recall the Jacobi triple product identity
    \(
        \sum_{k=-\infty}^{\infty} (-z)^{k}q^{k(k-1)/2} = \<z;q\>_{\infty} = (z,q/z,q;q)_{\infty} \qquad 
        \text{for $|q|<1$ and $z \neq 0$},
    \)
and recall several equivalent versions of the quintuple product identity, namely 
    \begin{multline}
    \label{qpideq}
        Q(z,q) = \sum_{k=-\infty}^{\infty} q^{k(3k-1)/2}z^{3k}(1-zq^k)
        = \<z;q\>_{\infty}(qz^2,q/z^2;q^2)_{\infty} \\
        = \<-qz^3;q^3\>_{\infty} - z\<-q^2z^3;q^3\>_{\infty},
    \end{multline}
again for $|q|<1$ and $z \neq 0$. For any $q$-series $\sum_{k \in \Z} a_{k}q^{k}$, the \emph{$r \mod{p}$ component} (or ``\emph{$r$-component}'', if $p$ is understood) of the series is $\sum_{k \in \Z} a_{kp+r} q^{kp+r}$.

For integers $t$, $x$, and $y$, with $y > 0$, it is elementary to verify the equalities
    \[
        (\pm q^{x+ty},\pm q^{y-x-ty};q^{y})_{\infty} 
        = (\mp 1)^{t} q^{-tx-\fr{t(t-1)y}{2}} (\pm q^{x},\pm q^{y-x};q^{y})_{\infty}.
    \]
Using this with the formulae $\<z;q\>_{\infty} = (z,q/z,q;q)_{\infty}$ and
    \[
        Q(z,q) = (z,q/z,q;q)_{\infty}(z^{2}q,q/z^{2};q^{2})_{\infty},
    \]
we easily deduce the following useful relations, which we record in a lemma for emphasis.

\begin{lemma}
    \label{lem:Shifts}
    For integers $t$, $x$, and $y$ with $y > 0$, one has
        \begin{subequations}
        \begin{align}
        \label{eq:jtpShift}
            \<\pm q^{x+ty};q^{y}\>_{\infty} &= (\mp 1)^{t} q^{-tx-\fr{t(t-1)y}{2}}\<\pm q^{x};q^{y}\>_{\infty}, \\
        \label{eq:QQShift}
            Q(q^{x+ty},q^{y}) &= q^{-3tx-\fr{t(3t-1)y}{2}}Q(q^{x},q^{y}).
        \end{align}
        \end{subequations}
\end{lemma}

\begin{corollary}
    \label{cor:<qx,qy>=0}
    If $y > 0$ and $x \equiv 0 \mod{y}$, then $\<q^{x};q^{y}\>_{\infty} = Q(q^{x},q^{y}) = 0$.
\end{corollary}

\begin{proof}
    Write $x=ty$, and apply Lemma \ref{lem:Shifts} and the fact that \mbox{$(1;q^{y})_{\infty} = 0$}.
\end{proof}

The following is a special case of a result due to Liu and Yang.

\begin{lemma}[{\cite[Thm.~2]{LY09}}]\label{lwt1}
Let $p = m^{2}+n^{2}$ with $m$ and $n$ coprime. For $|q|<1$, one has
    \(\label{lwt1eq1}
        \<-uq;q^2\>_{\infty} \<-vq;q^2\>_{\infty} 
        = \sum_{k=0}^{p-1} q^{k^2}u^{k} \<-u^{m}v^{n}q^{2km+p};q^{2p}\>_{\infty} \<-u^{n}v^{-m}q^{2kn+p};q^{2p}\>_{\infty}.
    \)
\end{lemma}

We aim to use Lemma \ref{lwt1} to rewrite the products of the form
    \[
        \<-q^{p\pm 3bm};q^{3p}\>_{\infty} \<-q^{p\pm 3bn};q^{3p}\>_{\infty}
    \]
that appear in the expansion of
    \(
    \label{qpideqa}
        \begin{aligned}
            & Q(q^{bm},q^p)Q(q^{bn},q^p) \\ 
            &\quad = \Big[ \<-q^{p+3bm};q^{3p}\>_{\infty} - q^{bm}\<-q^{p-3bm};q^{3p}\>_{\infty}\Big]
            \Big[\<-q^{p+3bn};q^{3p}\>_{\infty} - q^{bn}\<-q^{p-3bn};q^{3p}\>_{\infty} \Big].
        \end{aligned}
    \)
In particular, if $p=m^{2}+n^{2}$, then replacing 
    \[
        q \to q^{\fr{3p}{2}}, \quad u \to q^{s}, \quad\text{and}\quad v \to q^{t},
    \]
equation \eqref{lwt1eq1} states that
    \(
    \label{eq:LY-Simplified}
    \begin{aligned}
        &\<-q^{s+\fr{3p}{2}};q^{3p}\>_{\infty} \<-q^{t+\fr{3p}{2}};q^{3p}\>_{\infty} \\
        &\qquad = \sum_{k=0}^{p-1} q^{sk+\fr{3k^2p}{2}} \<-q^{3kmp + (ms+nt) + \fr{3p^2}{2}};q^{3p^2}\>_{\infty}
            \<-q^{3knp + (ns-mt) + \fr{3p^2}{2}};q^{3p^2}\>_{\infty}.
    \end{aligned}
    \)
Using the pair
    \[
        (s_0, t_0) = \left(3bm-\fr{p}{2}, 3bn-\fr{p}{2}\right)
    \]
for $(s,t)$ in \eqref{eq:LY-Simplified} evidently produces $\<-q^{p+3bm};q^{3p}\>_{\infty}\<-q^{p+3bn};q^{3p}\>_{\infty}$ on the left-hand side there. For the product $\<-q^{p+3bm};q^{3p}\>_{\infty}\<-q^{p-3bn};q^{3p}\>_{\infty}$, we use the pair $(s_0,t_{0}+p)$; indeed, in this case we have
    \[
        \<-q^{p-3bn};q^{3p}\>_{\infty}
        = \<-q^{3bn+2p};q^{3p}\>_{\infty}
        = \<-q^{t_0+p+\fr{3p}{2}};q^{3p}\>_{\infty}.
    \]
Continuing along this line, in total we find that the pairs
    \[
        (s_{0},t_{0}), \quad (s_{0}+p,t_{0}), \quad (s_{0},t_{0}+p), \quad\text{and}\quad (s_{0}+p,t_{0}+p)
    \]
produce the products
    \begin{alignat*}{2}
    &   \<-q^{3bm+p},-q^{3bn+p};q^{3p}\>_{\infty}, \qquad
        &&\<-q^{3bm+2p},-q^{3bn+p};q^{3p}\>_{\infty}, \\
    &   \<-q^{3bm+p},-q^{3bn+2p};q^{3p}\>_{\infty}, \quad\text{and}\quad
        &&\<-q^{3bm+2p},-q^{3bn+2p};q^{3p}\>_{\infty},
    \end{alignat*}
respectively.

Throughout the remainder of this paper we frequently use the following definitions.

\begin{definition}
\label{def:y-mu-s}
For fixed $p \equiv 1 \mod{4}$ with $p = m^{2}+n^{2}$, let
    \begin{subequations}
    \[
        y := 3p^{2},
    \]
let $\mu$ be a fixed solution to
    \[
        3m\mu \equiv 1 \pmod{p},
    \]
and let
    \[
        s := (3m\mu-1)/p.
    \]
    \end{subequations}
\end{definition}

\begin{definition}
    For fixed $p \equiv 1 \mod{4}$ with $p = m^{2}+n^{2}$ and $b > 0$, let
    \begin{subequations}
        \begin{align}
        \label{eq:a(k)}
            \xa_{k} &:= 
            3bp + 3kmp - \tfr{(m+n)p}{2} + \tfr{3p^2}{2}, \\
        \label{eq:b(k)}
            \xb_{k} &:= 
            3knp + \tfr{(m-n)p}{2} + \tfr{3p^2}{2}, \\
        \label{eq:g(k)}
            \xg_{k} &:= 3bkm + \tf{k(3k-1)p}{2}.
        \end{align}
    \end{subequations}
\end{definition}

The following is an immediate consequence of Lemma \ref{lwt1}.

\begin{corollary}\label{expancor}
The following expansions hold:
    \begin{subequations}
    \begin{align}
    \label{qtppexpansca}
        & \<-q^{3bm+p},-q^{3bn+p};q^{3p}\>_{\infty} 
        =   \sum_{k=0}^{p-1} q^{\xg_{k}} \<-q^{\xa_{k}},-q^{\xb_{k}};q^{y}\>_{\infty}, \\
    %==========
    \label{qtppexpanscb}
        &   \<-q^{3bm+p},-q^{3bn+2p};q^{3p}\>_{\infty}
        =   \sum_{k=0}^{p-1}q^{\xg_{k}} \<-q^{\xa_{k}+np},-q^{\xb_{k}-mp};q^{y}\>_{\infty}, \\
    %==========
    \label{qtppexpanscc}
    &   \<-q^{3bm+2p},-q^{3bn+p};q^{3p}\>_{\infty}
    =   \sum_{k=0}^{p-1}q^{\xg_{k}+pk} \<-q^{\xa_{k}+mp},-q^{\xb_{k}+np};q^{y}\>_{\infty} \\
    %==========
    \label{qtppexpanscd}
    &   \<-q^{3bm+2p},-q^{3bn+2p};q^{3p}\>_{\infty}
    =   \sum_{k=0}^{p-1}q^{\xg_{k}+pk} \<-q^{\xa_{k}+(m+n)p},-q^{\xb_{k}-(m-n)p};q^{y}\>_{\infty}.
    \end{align}
    \end{subequations}
\end{corollary}

\begin{definition}
    For all integers $t$, let
    \[
        \xs_{t} = \fr{t(t-1)y}{2} \qquad (y = 3p^2).
    \]
\end{definition}

With the above definition, equation \eqref{eq:jtpShift} states that
    \[
        \<\pm q^{x+ty};q^{y}\>_{\infty} 
        = (\mp 1)^{t} q^{-tx-\xs_{t}} \<\pm q^{x};q^{y}\>_{\infty}.
    \]
Moreover, it is helpful to note that $\xs_{-t} = ty+\xs_{t}$, so that
    \[
        \<\pm q^{x-ty};q^{y}\>_{\infty} 
        = (\mp 1)^{t} q^{t(x-y)-\xs_{t}} \<\pm q^{x}; q^{y}\>_{\infty}.
    \]

\begin{corollary}
\label{k+p}
    The summands on the right-hand sides of equations \eqref{qtppexpansca}--\eqref{qtppexpanscd} are all invariant under the transformations $k \to k \pm p$.
\end{corollary}

\begin{proof}
We demonstrate the invariance for \eqref{qtppexpansca} and the shift $k \to k+p$. 
Changing $k$ to $k+p$, we directly find that
    \begin{align*}
        \xa_{k+p} &= \xa_{k} + 3mp^{2} = \xa_{k} + my, \\
        \xb_{k+p} &= \xb_{k} + 3np^{2} = \xb_{k} + ny,
    \end{align*}
so that
    \begin{align*}
        &\<-q^{\xa_{k+p}},-q^{\xb_{k+p}};q^{y}\>_{\infty}
        = \<-q^{\xa_{k}+my},-q^{\xb_{k}+ny};q^{y}\>_{\infty}
        = q^{-m\xa_{k}-\xs_{m} -n\xb_{k} -\xs_{n}} 
        \<-q^{\xa_{k}},-q^{\xb_{k}};q^{y}\>_{\infty}.
    \end{align*}
Now observing that 
    \[
        \xg_{k+p} = \xg_{k} + \xg_{p} + 3kp^{2},
    \]
we simply compute that
    \(
        \xg_{k+p} - m\xa_{k}-\xs_{m} - n\xb_{k}-\xs_{n}
        = \xg_{k} + \tf{1}{2}p(3p+6k-1)(p-m^{2}-n^{2}) = \xg_{k}, 
    \)
and deduce that
    \[
        q^{\xg_{k+p}}\<-q^{\xa_{k+p}},-q^{\xb_{k+p}};q^{y}\>_{\infty} = q^{\xg_{k}}\<-q^{\xa_{k}},-q^{\xb_{k}};q^{y}\>_{\infty}.
    \]
The invariance of \eqref{qtppexpansca} under $k \to k - p$ and the analogous results for \eqref{qtppexpanscb}--\eqref{qtppexpanscd} are proved similarly.
\end{proof}

%%%%%%%%%%%%%%%%%%%%%%%%%%%%%%%%%%%%%%%%%%
\subsection{A key dissection lemma}
\label{sec:keylem}
%%%%%%%%%%%%%%%%%%%%%%%%%%%%%%%%%%%%%%%%%%

From \eqref{qpideqa} and the definition of $(a_{t})_{t}$, it is clear that
    \(
    \label{QQa}
    \begin{aligned}
        \sum_{t=-\infty}^{\infty}a_{t}q^{t} 
        &=  \<-q^{3bm+p},-q^{3bn+p};q^{3p}\>_{\infty} 
            - q^{bn}\<-q^{3bm+p},-q^{3bn+2p};q^{3p}\>_{\infty} \\
        & \quad -q^{bm}\<-q^{3bm+2p},-q^{3bn+p};q^{3p}\>_{\infty} 
        +   q^{bm+bn}\<-q^{3bm+2p},-q^{3bn+2p};q^{3p}\>_{\infty}.
    \end{aligned}
    \)

\begin{lemma}
    \label{lem:rComponent}
    Fix $r$ with $0 \leq r \leq p-1$, let $\xk$ satisfy
        \[
            3b\xk m \equiv r \mod{p}, \qquad 0 \leq \xk \leq p-1,
        \]
    and for this fixed $\xk$ let
        \[
            \xa = \xa_{\xk}, \quad \xb = \xb_{\xk}, \quad\text{and}\quad \xg = \xg_{\xk}.
        \]
    Then the $r$-component of $Q(q^{bm},q^{p})Q(q^{bn},q^{p})$, namely $\sum_{t\in\Z} a_{pt+r}q^{pt+r}$, is equal to
        \(
        \label{eq:rComponent0}
        \begin{aligned}
            & q^{\xg}\<-q^{\xa},-q^{\xb};q^{y}\>_{\infty} - q^{\xg+\xd}\<-q^{\xa-nsp^{2}},-q^{\xb+msp^{2}};q^{y}\>_{\infty} \\
            & \qquad  -q^{\xg+\xi}\<-q^{\xa-msp^{2}},-q^{\xb-nsp^{2}};q^{y}\>_{\infty} + q^{\xg+\xd+\xi}\<-q^{\xa-msp^{2}-nsp^{2}},-q^{\xb+msp^{2}-nsp^{2}};q^{y}\>_{\infty},
        \end{aligned}
        \)
    where
        \begin{align*}
            \xd &= -bnps + \tf{1}{6}(1+ps)p^{2}s, \\
            \xi &= -bmps-\xk p^2{s} + \tf{1}{6}(1+ps)p^{2}s.
        \end{align*}
\end{lemma}

\begin{proof}
The equations of Corollary \ref{expancor} yield that
    \begin{align*}
        \sum_{t=-\infty}^{\infty}a_{t}q^{t} &= 
        \sum_{k} q^{\xg_{k}} \<-q^{\xa_{k}},-q^{\xb_{k}};q^{y}\>_{\infty} 
        - \sum_{k} q^{\xg_{k}+bn} \<-q^{\xa_{k}+np},-q^{\xb_{k}-mp};q^{y}\>_{\infty}  \\
        & \qquad - \sum_{k} q^{\xg_{k}+bm+pk} \<-q^{\xa_{k}+mp},-q^{\xb_{k}+np};q^{y}\>_{\infty}\\
        & \qquad + \sum_{k} q^{\xg_{k}+b(m+n)+pk} \<-q^{\xa_{k}+(m+n)p},-q^{\xb_{k}+(m-n)p};q^{y}\>_{\infty} \\
        &=: \xS_{1} - \xS_{2} - \xS_{3} + \xS_{4},
    \end{align*}
and within each $\xS_{i}$, the $r$-component is determined by some $0 \leq k_{i} \leq p-1$ that makes the ``external $q$-exponent'' equivalent to $r \mod{p}$\footnote{We note from \eqref{eq:a(k)} and \eqref{eq:b(k)} that $p \.|\. \xa_{k}$ and $p \.|\. \xb_{k}$, respectively.}. Specifically, since $\xg_k \equiv 3bkm \mod{p}$, for $1 \leq i \leq 4$ the $r$-component of $\xS_{i}$ corresponds to $0 \leq k_{i} \leq p-1$ satisfying
    \begin{subequations}
    \begin{align}
    \label{eq:gamma(k1)}
        3bk_{1}m &\equiv r \mod{p}, \\
        bm+3bk_{2}m &\equiv r \mod{p}, \\
        bn+3bk_{3}m &\equiv r \mod{p}, \\
    \label{eq:gamma(k4)}
        b(m+n)+3bk_{4}m &\equiv r \mod{p},
    \end{align}
    \end{subequations}
respectively. Fixing $\xk$ such that
    \[
        3b\xk m \equiv r \mod{p}
    \]
then, the quantities
    \(
    \label{eq:kSols}
        k_{1} = \xk, \quad 
        k_{2} = \xk-n\mu, \quad
        k_{3} = \xk-m\mu, \quad\text{and}\quad
        k_{4} = \xk-m\mu-n\mu
    \)
satisfy \eqref{eq:gamma(k1)}--\eqref{eq:gamma(k4)}, respectively, and with this fixed $\xk$ we let
    \[
        \xa=\xa_{\xk}, \quad \xb=\xb_{\xk}, \quad\text{and}\quad \xg=\xg_{\xk}.
    \]

It remains now to simply expand the $k$-th terms in each $\xS_{i}$ using the respective $k_{i}$ from \eqref{eq:kSols}. 
There is nothing to do when $k = k_{1} = \xk$, so let $k = k_{2} = \xk-n\mu$ and consider the $k$-summand of $\xS_{2}$. We find that 
    \begin{align*}
        \xa_{\xk-n\mu}+np &= \xa-nsp^{2}, \\
        \xb_{\xk-n\mu}-mp &= \xb+msp^{2}-\mu y, \\
        bn+\xg_{\xk-n\mu} &= \xg+\xd_{0},
    \end{align*}
where
    \[
        \xd_{0} := -bnps + (1-3\xk)\mu n + \tfr{1}{2}\mu n(3\mu n-1)p.
    \]
Thus, the $r$-component of $\xS_{2}$ is
    \[
        q^{\xg+\xd_{0}}\<-q^{\xa-nsp^{2}},-q^{\xb+msp^{2}-\mu y};q^{y}\>_{\infty},
    \]
and we apply \eqref{eq:jtpShift} to simplify this to
    \(
        q^{\xg+\xd}\<-q^{\xa-nsp^{2}},-q^{\xb+msp^{2}};q^{y}\>_{\infty}, \qquad\text{with}\qquad 
        \xd = -bnps + \fr{(1+ps)p^{2}s}{6}.
    \)
The analogous formulae for $\xS_{3}$ and $\xS_{4}$ follow from similar direct computations.
\end{proof}

%%%%%%%%%%%%%%%%%%%%%%%%%%%%%%%%%%%%%%%%%%
\section{The case \tops{$p\equiv 1\mod{12}$}{p≡ 1(mod12)}}
%%%%%%%%%%%%%%%%%%%%%%%%%%%%%%%%%%%%%%%%%%
\label{sec:1(mod12)}

The assumption that $p \equiv 1 \mod{12}$, in particular the assumption that $p \equiv 1 \mod{3}$, implies that either $m$ or $n$ must be $0 \mod{3}$. As our arguments (previous and following) are symmetric in $m$ and $n$, no generality is lost in assuming that $n \equiv 0 \mod{3}$. With this assumption, in this section we set
    \[
        N := \fr{ns}{3},
    \]
and we recall that $3m\mu = 1 + ps$. The goal of this section is proof of the following theorem.

\begin{theorem}\label{Qdissect1}
    Let $p \equiv 1 \pmod{12}$ and $p=m^2+n^2$, with~$3|n$, and let $b > 0$.
    \begin{enumerate}[label={\normalfont(\arabic*)}]
    \item \label{item:Thm(1)} If $m \equiv 1 \pmod{3}$, then 
        \(\label{Qdissect1eq1}
            Q(q^{b m},q^p)Q(q^{b n},q^p) 
            = \sum_{k=0}^{p-1} q^{\xg_{k}}
                Q\?\big({q^{(\xa_{k}-p^2)/3},q^{p^2}}\big) 
                Q\?\big({q^{(2p^{2} - \xb_{k})/3},q^{p^2}}\big).
        \)
    \item \label{item:Thm(2)} If $m \equiv 2 \pmod{3}$, then 
        \(\label{Qdissect1eq2}
            Q(q^{bm},q^p)Q(q^{bn},q^p)
            = \sum_{k=0}^{p-1} q^{\xg_{k}} 
                Q\?\big({ q^{(2p^{2}-\xa_{k})/3},q^{p^2}}\big) 
                Q\?\big({q^{(\xb_{k}-p^{2})/3},q^{p^2}}\big).
        \)
    \end{enumerate}
\end{theorem}

\begin{proof}
In short, we aim to show that all $r$-components (for $0 \leq r \leq p-1$) of both sides of \eqref{Qdissect1eq1} and \eqref{Qdissect1eq2} are equal. Thus, fix $r$ with $0 \leq r \leq p-1$, let $\xk$ satisfy
    \[
        3b\xk m \equiv r \mod{p}, \qquad 0 \leq \xk \leq p-1,
    \]
and for this $\xk$ let
    \[
        \xa = \xa_{\xk}, \quad \xb = \xb_{\xk}, \quad\text{and}\quad \xg = \xg_{\xk}.
    \]

By Lemma \ref{lem:rComponent}, the $r$-component of $Q(q^{bm},q^{p})Q(q^{bn},q^{p})$ is
    \(
    \label{eq:rCompRecall}
    \begin{aligned}
        & q^{\xg}\<-q^{\xa},-q^{\xb};q^{y}\>_{\infty} - q^{\xg+\xd}\<-q^{\xa-nsp^{2}},-q^{\xb+msp^{2}};q^{y}\>_{\infty} \\
        & \qquad - q^{\xg+\xi}\<-q^{\xa-msp^{2}},-q^{\xb-nsp^{2}};q^{y}\>_{\infty} \\
        &\qquad + q^{\xg+\xd+\xi}\<-q^{\xa-msp^{2}-nsp^{2}},-q^{\xb+msp^{2}-nsp^{2}};q^{y}\>_{\infty},
    \end{aligned}
    \)
where
    \[
        \xd = -bnsp + \tf{1}{6}(1+ps)p^{2}s \qquad\text{and}\qquad \xi = -bmps-\xk p^2{s} + \tf{1}{6}(1+ps)p^{2}s.
    \]
Under our assumptions that $3|n$ and $N = \tf{ns}{3}$, 
we may further reduce \eqref{eq:rCompRecall}. In particular, as $nsp^2 = Ny$ and $\xs_{-N} = Ny+\xs_{N}$, we use \eqref{eq:jtpShift} to transform \eqref{eq:rCompRecall} into
    \[
    \begin{aligned}
        & q^{\xg}\<-q^{\xa},-q^{\xb};q^{y}\>_{\infty} 
        - q^{\xg+\xd+[N(\xa-y)-\xs_{N}]}\<-q^{\xa},-q^{\xb+msp^{2}};q^{y}\>_{\infty} \\
        & \qquad  -q^{\xg+\xi+[N(\xb-y)-\xs_{N}]}\<-q^{\xa-msp^{2}},-q^{\xb};q^{y}\>_{\infty} \\
        &\qquad + q^{%
            \xg+\xd+\xi+\left[N(\xa-msp^{2}-y)-\xs_{N}\right]
            + \left[N(\xb+msp^{2}-y)-\xs_{N}\right]%
        } %
        \<-q^{\xa-msp^{2}},-q^{\xb+msp^{2}};q^{y}\>_{\infty},
    \end{aligned}
    \]
which we see can be factored as
    \(
    \label{eq:rComp0Factored}
    \begin{aligned}
        q^{\xg} &%
        \left(%
            \<-q^{\xa};q^{y}\>_{\infty} - q^{\xi+[N(\xb-y)-\xs_{N}]}\<-q^{\xa-msp^{2}};q^{y}\>_{\infty}%
        \right) \\
        &\qquad\times \left(%
            \<-q^{\xb};q^{y}\>_{\infty} - q^{\xd+[N(\xa-y)-\xs_{N}]}\<-q^{\xb+msp^{2}};q^{y}\>_{\infty}%
        \right).
    \end{aligned}
    \)

Pausing to consider the right-hand sides of \eqref{Qdissect1eq1} and \eqref{Qdissect1eq2}, we recall that
    \[
        Q(z,q) = \<-z^{3}q;q^{3}\>_{\infty} - z\<-z^{3}q^{2};q^{3}\>_{\infty},
    \]
so that, for general integers $A$ and $B$, one has
    \(
    \label{eq:QQExpansionDemo}
    \begin{aligned}
        Q(q^{\fr{A}{3}},q^{p^2})Q(q^{\fr{B}{3}},q^{p^2}) 
        &= \left(\<-q^{A+p^2};q^y\>_{\infty} - q^{\fr{A}{3}}\<-q^{A+2p^2};q^{y}\>_{\infty}\right) \\
        &\qquad \times \left(\<-q^{B+p^2};q^y\>_{\infty} - q^{\fr{B}{3}}\<-q^{B+2p^2};q^{y}\>_{\infty}\right)
    \end{aligned}
    \qquad (y=3p^2).
    \)
Comparing \eqref{eq:rComp0Factored} and \eqref{eq:QQExpansionDemo} (ignoring the extra $q^{\xg}$ in the former), we want to select $A$ and $B$ to both equate the expressions, and to have $A/3$ and $B/3$ be integral. To equate \eqref{eq:rComp0Factored} and \eqref{eq:QQExpansionDemo} then, evidently we want
    \[
        A+p^2 = \xa \qquad\text{or}\qquad A+p^2 = 3p^2 - \xa,
    \]
and
    \[
        B+p^2 = \xb \qquad\text{or}\qquad B+p^2 = 3p^2 - \xb.
    \]
For our integrality condition, dividing these potential $A$ and $B$ by $3$, we see that
    \[
        \fr{\xa-p^2}{3} = bp+\xk mp+\fr{(p-m-n)p}{6} \qquad\text{and}\qquad
        \fr{2p^2-\xa}{3} = -bp-\xk mp+\fr{(p+m+n)p}{6},
    \]
and
    \[
        \fr{\xb-p^2}{3} = \xk np+\fr{(p+m-n)p}{6} \qquad\text{and}\qquad
        \fr{2p^2-\xb}{3} = -\xk np+\fr{(p-m+n)p}{6}.
    \]

Thus, our choices for $A$ and $B$ must ensure that at least one of the numerators in each pair is $0 \mod{6}$. Already we know that $p \pm m \pm n \equiv 0 \mod{2}$, so it remains only to consider $p \pm m \pm n \mod{3}$. We have assumed that $p \equiv 1 \mod{3}$ and $n \equiv 0 \mod{3}$, so our choices for $A$ and $B$ depend only on $m \mod{3}$. 

Focusing on the case $m \equiv 1 \mod{3}$, we see at once that
    \[
        \fr{\xa-p^2}{3} = bp+\xk mp+\fr{(p-m-n)p}{6} \qquad\text{and}\qquad 
        \fr{2p^2-\xb}{3} = -\xk n + \fr{(p-m+n)p}{6}
    \]
are integral. Thus, in this case we let
    \[
        A = \xa-p^2 \qquad\text{and}\qquad 
        B = 2p^2-\xb,
    \] 
and, since $s \equiv 2 \mod{3}$, let
    \[
        ms = 2 + 3M \qquad (\text{cf.~$ns = 3N$}),
    \]
and directly verify that \eqref{eq:rComp0Factored} and \eqref{eq:QQExpansionDemo} are equal. The computations are tedious and cumbersome: as an example we note that
    \[
        \<-q^{\xa-msp^2};q^{y}\>_{\infty} 
        = \<-q^{(\xa+p^2)-(M+1)y};q^{y}\>_{\infty}
        = q^{\lf[(M+1)(\xa+p^2-y)-\xs_{M+1}\rh]} \<-q^{\xa+p^2};q^{y}\>_{\infty},
    \]
so that in \eqref{eq:rComp0Factored} we have
    \[
    q^{\xi+\lf[N(\beta-y)-\xs_{N}\rh]}
        \<-q^{\xa-msp^2}; q^{y}\>_{\infty} 
    = q^{\xi+\lf[N(\beta-y)-\xs_{N}\rh]+\lf[(M+1)(\xa+p^2-y)-\xs_{M+1}\rh]}
        \<-q^{A+2p^2};q^{y}\>_{\infty}.
    \]
One may then verify---using a symbolic calculator if desired---that indeed
    \(
        \frac{A}{3}  
        = \xi + \lf[N(\xb-y)-\xs_{N}\rh] 
        + \lf[(M+1)(\xa+p^2-y)-\xs_{M+1}\rh],
    \)
as desired, and then make an analogous verification for $B/3$. Thus, with the equality of \eqref{eq:rComp0Factored} and \eqref{eq:QQExpansionDemo} we deduce that
    \[
        \sum_{t=-\infty}^{\infty} a_{pt+r}q^{pt+r} = q^{\xg}%
        Q\?\big(q^{(\xa-p^2)/3},q^{p^2}\big) 
        Q\?\big(q^{(2p^2-\xb)/3},q^{p^2}\big),
    \]
and, repeating this for $0 \leq r \leq p-1$, we establish part \ref{item:Thm(1)} of the theorem.

For the case $m \equiv 2 \mod{3}$, we find that
    \[
        \fr{2p^2-\xa}{3} = -bp-kmp+\fr{(p+m+n)p}{6} \qquad\text{and}\qquad
        \fr{\xb-p^2}{3} = knp+\fr{(p+m-n)p}{6}
    \]
are integral, so that letting
    \[
        A = 2p^2-\xa \qquad\text{and}\qquad 
        B = \xb-p^2,
    \] 
one may similarly directly verify that \eqref{eq:rComp0Factored} and \eqref{eq:QQExpansionDemo} are equal, and thereby establish part \ref{item:Thm(2)} of the theorem.
\end{proof}

\begin{corollary}\label{Qdissectanvanc}
    Let $p \equiv 1 \mod{12}$, let $m$, $n$, and $b$ be as in Theorem \ref{Qdissect1}, let 
        \[
            w \equiv \bar{2}(m+n)\mod{p},
        \]
    and let $(a_{t})_{t \in \Z}$ be defined as in \eqref{QQa}. One has
        \[
            a_{pt+bw} = a_{pt+b(w-3b)}=0 \qquad \text{for all $t$.}
        \]
\end{corollary}

\begin{proof}
First suppose that $m \equiv 1 \mod{3}$. Recalling from Corollary \ref{cor:<qx,qy>=0} that \mbox{$Q(q^{x},q^{y}) = 0$} whenever $x \equiv 0 \mod{y}$, where here $y > 0$ can be a general integer, we aim to apply said corollary in \eqref{Qdissect1eq1}. By \eqref{Qdissect1eq1} and the definitions of $\xa_{\xk}$, $\xb_{\xk}$, and $\xg_{\xk}$, for any $0 \leq r \leq p-1$ the $r$-component of $Q(q^{bm},q^{p})Q(q^{bn},q^{p})$ is
    \(
    \label{eq:rComp(Temp)}
        \sum_{t=-\infty}^{\infty} a_{pt+r}q^{pt+r} = q^{3b\xk m+\fr{\xk(3\xk-1)p}{2}} 
        Q\?\big({q^{bp+\xk mp + \fr{(p-m-n)p}{6}},q^{p^2}}\big) 
        Q\?\big({q^{-\xk np + \fr{(p-m+n)p}{6}},q^{p^2}}),
    \) 
where $3b\xk m \equiv r \mod{p}$ and $0 \leq \xk \leq p-1$. 

We first show that $a_{pt+bw}=0$ for all $t$, so fix $r \equiv bw \mod{p}$ in \eqref{eq:rComp(Temp)}. From the relation $3b\xk m \equiv bw \mod{p}$ and the equality $p = m^2+n^2$, it is easy to check that
    \[
        6\xk n \equiv \bar{m}n(m+n) \equiv \bar{m}(mn-m^{2}) \equiv n-m \mod{p},
    \]
so let
    \(
    \label{eq:6kn=pu(p)}
        6\xk n = n-m+pK.
    \)
Then
    \(
    \label{eq:Q0Factor(Case1)}
        Q\?\big({q^{-\xk np+\fr{(p-m+n)p}{6}},q^{p^{2}}}\big) 
        = Q\?\big(q^{\fr{(1-K)p^{2}}{6}},q^{p^{2}}\big),
    \)
so to apply Corollary \ref{cor:<qx,qy>=0} it suffices to show that $K \equiv 1 \mod{6}$. From \eqref{eq:6kn=pu(p)} and our assumption on $p$, evidently $K \equiv m-n \mod{6}$, and since $m-n \equiv 1 \mod{2}$ (trivially) and $m-n \equiv m \equiv 1 \mod{3}$, it follows at once that $K\equiv 1 \mod{6}$; thus \eqref{eq:Q0Factor(Case1)} and \eqref{eq:rComp(Temp)} are identically $0$, as are all $a_{pt+bw}$. 

Now fixing $r \equiv b(w-3b) \mod{p}$ in \eqref{eq:rComp(Temp)}, we find that $6\xk m \equiv -6b+m+n \mod{p}$, and we write
    \(
    \label{eq:6km-pK}
        6\xk m = -6b+m+n+pK.
    \)
Then 
    \(
    \label{eq:Q0Factor(Case2)}
        Q\?\big({%
            q^{bp+\xk mp+\fr{(p-m-n)p}{6}},q^{p^{2}}
        }\big) 
        = Q\?({
            q^{\fr{(1+K)p^{2}}{6}},q^{p^{2}}
        }\big),
    \)
so in this case it suffices to have $K \equiv -1 \mod{6}$. Said congruence is easily verified from \eqref{eq:6km-pK} and our assumptions on $m$ and $n$, whereby  \eqref{eq:Q0Factor(Case2)}, and consequently \eqref{eq:rComp(Temp)}, are equal to $0$, as are all $a_{pt+b(w-3b)}$. 

Finally, when $m \equiv 2 \mod{3}$, the result follows from completely parallel arguments using \eqref{Qdissect1eq2} in place of \eqref{Qdissect1eq1}.
\end{proof}

\begin{example}\label{p13vancex}
Let $p=13=2^{2}+3^{3}$ and $b=5$ in Corollary \ref{Qdissectanvanc}, and define $(a_t)_{t\in\Z}$ via
    \[
        Q(q^{10},q^{13})Q(q^{15},q^{13}) = \sum_{t=0}^{\infty} a_t q^t.
    \]
Then $w \equiv 9 \mod{13}$, $bw \equiv 6 \mod{13}$, and $b(w-3b) \equiv 9 \mod{13}$, whereby 
    \[
        a_{13t+6} = a_{13t+9} = 0 \qquad (t \in \Z).
    \]
\end{example}

%%%%%%%%%%%%%%%%%%%%%%%%%%%%%%%%%%%%%%%%%%
\subsection{Sign patterns when \tops{$p\equiv 1\mod{12}$}{p≡1(mod 12)}}
%%%%%%%%%%%%%%%%%%%%%%%%%%%%%%%%%%%%%%%%%%
\label{sec:signs1(mod12)}

We recall that
    \(
    \label{eq:QQ(Recall)}
        Q(q^{\l},q^{p}) = (q^{\l},q^{p-\l},q^{p};q^{p})_{\infty}(q^{p+2\l},q^{p-2\l};q^{2p})_{\infty}.
    \)

\begin{lemma}
    \label{lem:PosCoeffs}
    If $p > 10$ and $0 < \l < \fr{p}{2}$, and one defines $(c_{t})_{t \geq 0}$ via
        \(
        \label{eq:Q(ql,qp)Pos}
            \fr{Q(q^{h},q^{p})}{(q;q)_{\infty}} = \sum_{k=0}^{\infty} c_{t} q^{t},
        \)
    then one has $c_{t} > 0$ for all $t > t_{0}(\l)$.
\end{lemma}

\begin{proof}
Writing $(q;q)_{\infty} = (q,q^{2},\ldots,q^{p};q^{p})_{\infty}$ and using the definition
    \[
        Q(q^{\l},q^{p}) = (q^{\l},q^{p-\l},q^{p};q^{p})_{\infty} (q^{p+2\l},q^{p-2\l};q^{2p})_{\infty},
    \]
we see that
    \(
    \label{eq:QQQuotient}
        \fr{Q(q^{h},q^{p})}{(q;q)_{\infty}}
        = (q^{2\l},q^{2p-2\l};q^{2p})_{\infty}^{-1}\prod_{\substack{k=1 \\[0.1em] k \,\not\equiv\, \pm r,\pm 2r \,(p)}}^{p-1} 
        (q^{k};q^{p})_{\infty}^{-1}.
    \)
Now, because $p \geq 11$, the range $\{1,\ldots,p-1\} \setminus \{\l,2\l,p-\l,p-2\l\}$ must contain at least one pair of consecutive numbers: if no such pair existed, then one must have $p \leq 10$. 
Letting $(x,y):=(x,x+1)$ be such a pair, the right hand side of \eqref{eq:QQQuotient} has a factor $(q^{x},q^{y};q^{p})_{\infty}^{-1}$, which itself has a factor
    \[
        (1-q^{x})^{-1}(1-q^{y})^{-1} = 1 + \sum_{k=1}^{\infty} \xr_{x,y}(k)q^{k},
    \]
where $\xr_{x,y}(k)$ counts the number of ways to write $k = ax+by$ with $a,b \geq 0$. As $(x,y)=1$ necessarily, it is well-known that $\xr_{x,y}(k) > 0$ for all $k > xy-x-y$. From this and the fact (from \eqref{eq:QQQuotient}) that the $c_{t}$ in \eqref{eq:Q(ql,qp)Pos} are at least \emph{nonnegative}, the result follows.
\end{proof}

To apply Lemma \ref{lem:PosCoeffs} to general quantities $Q(q^{h},q^{p})$, we first relate them to quantities of the form $Q(q^{\l},q^{p})$ with $0 \leq \l \leq \fr{p}{2}$. 
Fixing $p > 1$ and letting $h$ be any integer, write
    \[
        h = kp + \l \qquad\text{with}\qquad 0 \leq \l < p,
    \]
and, by the ``shift relation'' \eqref{eq:QQShift}, we have
    \[
        Q(q^{h},q^{p}) = q^{H}Q(q^{\l},q^{p}) \qquad 
        (H = -3k\l-\tf{k(3k-1)p}{2}).
    \]
If $0 \leq \l \leq \fr{p}{2}$ already then there is nothing more to change; if $\l > \fr{p}{2}$, then we write 
    \[
        Q(q^{h},q^{p}) = q^{H}Q(q^{\l},q^{p}) = (-1)q^{H+p-2\l}Q(q^{p-\l},q^{p}) \qquad (H = -3k\l-\tf{k(3k-1)p}{2}),
    \]
and $0 < p-\l < \fr{p}{2}$. Because $\l + (p-2\l) = p-\l$ and $(-1)^{p-2\l} = -1$, we may combine these cases in the following formula.

\begin{lemma}
    \label{lem:Q-rReduction2}
    Let $h$ and $p$ be integral with $p > 1$, write
        \[
            h = kp + \l \qquad \text{with}\qquad  0 \leq \l < p,
        \]
    and define
        \[
            \xe = \begin{cases}
            0, & 0 \leq \l \leq \fr{p}{2}, \\
            p-2\l, & \fr{p}{2} < \l < p.
            \end{cases}
        \]
    Then
        \[
            Q(q^{h},q^{p}) = (-1)^{\xe} q^{H+\xe} Q(q^{\l+\xe},q^{p}),  
        \]
    where $H = -3k\l - \fr{k(3k-1)p}{2}$.
\end{lemma}

{ % renewing \. spacing
\renewcommand{\.}{\mspace{0.9mu}}

\begin{corollary}\label{Qdissectanvanceqsign1}
    Let $p \equiv 1 \pmod{12}$ and $p=m^{2}+n^{2}$, with $3 \?\mid\? n$. Let $b>0$ with $(b,p)=1$, and define the sequence $(b_{t})_{t\in\Z}$ by
        \(\label{Qdissectc2eq1}
            \frac{Q(q^{bm},q^{p})Q(q^{bn},q^p)}{(q^p;q^p)_{\infty}^2} 
            = \sum_{t=-\infty}^{\infty} b_{t} q^{t}.
        \)
    Fix $0 \leq r \leq p-1$ and $\xk$ satisfying
        \[
            3b\xk m \equiv r \mod{p}, \qquad 0 \leq \xk \leq p-1.
        \]     
If $m \equiv 1 \pmod{3}$, let 
    \begin{align*}
        h &:= \fr{\xa_{\xk}-p^2}{3p} = b+\xk m+\tf{1}{6}(p-m-n), \\[0.2em]
        h\.' &:= \fr{2p^{2}-\xb_{\xk}}{3p} = -\xk n + \tf{1}{6}(p-m+n),
    \end{align*}
and if $m \equiv 2 \pmod{3}$, let
    \begin{align*}
        h &:= \fr{2p^{2}-\xa_{\xk}}{3p} = -b-\xk m+\tf{1}{6}(p+m+n), \\[0.2em]
        h' &:= \fr{\xb_{\xk}-p^{2}}{3p} = \xk n + \tfr{1}{6}(p+m-n).
    \end{align*}
In addition, define $k$ and $\l$ so that
    \[
        h = kp+\l \qquad \text{with $0 \leq \l < p$},
    \]
let
    \[
        \xe := \begin{cases}
            0 & 0 \leq \l < \fr{p}{2}, \\
            p-2\l & \fr{p}{2} < \l < p,
        \end{cases}
    \]
and similarly define $k'$, $\l\.'$, and $\xe'$.
Then either $b_{pt+r} = 0$ for all $t$, or one has 
    \[
        \sgn b_{pt+r} = (-1)^{\xe+\xe'} \qquad \text{for $t > t_{0}(r)$.}
    \]
\end{corollary}

\begin{proof}
    Fixing $0 \leq r < p$ and $\xk$ such that $3b\xk m \equiv r \mod{p}$, suppose that \mbox{$m \equiv 1 \mod{3}$}, and define $h$, $h\.'$, $k$, $k'$, etc.~as described above. In light of Corollary \ref{Qdissectanvanc}, we further assume that $r$ is not congruent to $bw$ or $b(w-3b)$ modulo $p$, where $w \equiv \bar{2}(m+n) \mod{p}$, so that $(b_{pt+r})_{t}$ is not identically zero. 
    
    By Theorem \ref{Qdissect1}, the $r$-component of \eqref{Qdissectc2eq1} is
        \[
        \begin{aligned}
        \sum_{t=-\infty}^{\infty} b_{pt+r} q^{pt+r}
        &=  
            q^{\xg_{\xk}} 
            \frac{Q(q^{(\xa_{\xk}-p^{2})/3},q^{p^2})}{(q^p;q^p)_{\infty}} \times \frac{Q(q^{(2p^{2}-\xb_{\xk})/3},q^{p^2})}{(q^p;q^p)_{\infty}} \\
        &=
            q^{\xg_{\xk}} 
            \frac{Q(q^{ph},q^{p^2})}{(q^p;q^p)_{\infty}} \times
            \frac{Q(q^{ph\.'},q^{p^2})}{(q^p;q^p)_{\infty}},
        \end{aligned}
        \]
    and, according to Lemma \ref{lem:Q-rReduction2}, one has
        \(
        \label{eq:Q(ql,qp)Quotient}
            \fr{Q(q^{h},q^{p})}{(q;q)_{\infty}} = (-1)^{\xe} q^{H+\xe} \fr{Q(q^{\l+\xe},q^{p})}{(q;q)_{\infty}} \qquad 
            \lf(H = -3k\l - \tfr{k(3k-1)p}{2}\rh).
        \)
    Then, since $0 < \l + \xe < \fr{p}{2}$ by design, Lemma \ref{lem:PosCoeffs} implies that the series coefficients of $Q(q^{\l+\xe},q^{p})/(q;q)_{\infty}$, say $(c_{t})_{t \geq 0}$, are all positive for $t \geq t_{0}$. Replacing $q$ with $q^{p}$ in \eqref{eq:Q(ql,qp)Quotient}, and making similar arguments for $Q(q^{h\.'},q^{p})/(q;q)_{\infty}$, we deduce that
        \[
            \sum_{t=-\infty}^{\infty} b_{pt+r} q^{pt+r}
            =
            (-1)^{\xe+\xe'} q^{\xg_{\xk}+p(H+H'+\xe+\xe')} 
            \fr{
            Q\?\big({q^{p(\l+\xe)},q^{p^2}}\big)
            Q\?\big({q^{p(\l\.'+\xe')},q^{p^2}}\big)
            }{
            (q^p;q^p)_{\infty}^2
            },
        \]
    and the result follows.
\end{proof}

} % end scope of \. change

\begin{example}\label{p13signpat}
For $p = 13 = 2^{2} + 3^{2}$ and $b=5$, the signs of $b_{13t+r}$ for $0 \leq r < 13$ and $t > t_{0}(r)$ are given in the following table.
    \[
    \begin{array}{c|ccccccccccccc}
        r& 0 & 1 & 2 & 3 & 4 & 5 & 6 & 7 & 8 & 9 & 10 & 11 & 12 \\
        \hline
        \operatorname{sgn} b_{13t+r} & - & + & + & + & - & - & 0 & + & + & 0 & + & + & - \\
    \end{array}
    \]
We note that the entries of ``$0$'' for $r=6$ and $r=9$ are consistent with Example \ref{p13vancex}.
\end{example}

%%%%%%%%%%%%%%%%%%%%%%%%%%%%%%%%%%%%%%%%%%
%%%%%%%%%%%%%%%%%%%%%%%%%%%%%%%%%%%%%%%%%%
\section{The case \texorpdfstring{$p \equiv 5 \pmod{12}$}{p≡5(mod 12)}} 
\label{sec:5(mod12)}
%%%%%%%%%%%%%%%%%%%%%%%%%%%%%%%%%%%%%%%%%%
%%%%%%%%%%%%%%%%%%%%%%%%%%%%%%%%%%%%%%%%%%

We first recall that for complex $a$, $b$, and $q$, with $a,b \neq 0$ and $|q|<1$, we let
    \[
        W(a,b;q) := \<a,b,ab,\tf{a}{b};q\>_{\infty}(q;q)_{\infty}^{-2}.
    \]
Our goal in this section is the proof of the following dissection result.

\begin{theorem}
    \label{thm:QuintDiss(5mod12)}
    Let $p \equiv 5 \mod{12}$ with $p = m^2+n^2$, and let $b > 0$. 
        \begin{enumerate}[label={\normalfont(\arabic*)}]
            \item \label{item:5(mod12)(1)}If $m \equiv -n \mod{3}$, then defining
                \[
                    \xc_{k} := \xg_{k} + p(1-ns)\big({kn+\tfr{p+m-n(1+ps)}{6}}\big),
                \]
            one has
                \(
                \label{eq:WinqDiss(m=-n(3))}
                    Q(q^{bm},q^{p})Q(q^{bn},q^{p}) = \sum_{k=0}^{p-1} q^{\xc_{k}}
                    W(-q^{\xa_{k}/3},-q^{(\xb_{k}-nsp^2)/3};q^{p^2}).
                \)
            \item \label{item:5(mod12)(2)} If $m \equiv n \mod{3}$, then defining
                \[
                    \xc_{k}^{*} := \xg_{k} + p(1-ms)\big({b+km+\tfr{p-m(1+ps)-n}{6}}\big),
                \]
            one has
                \(
                \label{eq:WinqDiss(m=n(3))}
                    Q(q^{bm},q^{p})Q(q^{bn},q^{p}) = \sum_{k=0}^{p-1} q^{\xc_{k}^{*}}
                    W(-q^{\xb_{k}/3},-q^{(\xa_{k}-msp^2)/3};q^{p^2}).
                \)
        \end{enumerate}
\end{theorem}

A key ingredient in our proof of Theorem \ref{thm:QuintDiss(5mod12)} is \emph{Winquist's identity}, which we recall in the following lemma.

\begin{lemma}[Winquist \cite{W69}]
For complex $a$, $b$, and $q$, with $a,b \neq 0$ and $|q|<1$, one has
    \(\label{wineq}
    W(a,b;q) = 
        \<a^3,b^3q;q^3\>_{\infty} 
        - b\<a^3,b^3q^2;q^3\>_{\infty} 
        - \fr{a}{b}\<a^3q,b^3;q^3\>_{\infty} 
        + \fr{a^2}{b}\<a^3q^2,b^3;q^3\>_{\infty}.
    \)
\end{lemma}

Fixing $0 \leq r \leq p-1$ and $3b\xk m \equiv r \mod{p}$, and letting $\xa=\xa_{\xk}$, $\xb=\xb_{\xk}$, etc., we recall from Lemma \ref{lem:rComponent} that the $r$-component of $Q(q^{bm},q^{p})Q(q^{bn},q^{p})$ is
    \(
    \label{eq:rComp(Recall)}
    \begin{aligned}
        &\sum_{t=-\infty}^{\infty} a_{pt+r}q^{pt+r} 
        = q^{\xg}\<-q^{\xa},-q^{\xb};q^y\>_{\infty}  
        - q^{\xg+\xd}\<-q^{\xa-nsp^{2}},-q^{\xb+msp^{2}};q^y\>_{\infty} \\
        &\qquad -q^{\xg+\xi}\<-q^{\xa-msp^{2}},-q^{\xb-nsp^{2}};q^y\>_{\infty}  
        + q^{\xg+\xd+\xi}\<-q^{\xa-(m+n)sp^{2}},-q^{\xb+msp^{2}-nsp^{2}};q^y\>_{\infty}.
    \end{aligned}
    \)
Next, for general $A$ and $B$ Winquist's identity \eqref{wineq} yields that
    \(
    \label{eq:WinqDemo}
    \begin{aligned}
    W(-q^{\fr{A}{3}},-q^{\fr{B}{3}};q^{p^2}) 
        &= \<-q^A,-q^{B+p^2};q^y\>_{\infty} + q^{\fr{B}{3}}\<-q^A,-q^{B+2p^2};q^y\>_{\infty} \\
        &\qquad -q^{\fr{A-B}{3}}\<-q^{A+p^2},-q^B;q^y\>_{\infty} - q^{\fr{2A-B}{3}}\<-q^{A+2p^2},-q^B;q^y\>_{\infty}. \\
    \end{aligned}
    \)
In the same spirit as the proof of Theorem \ref{Qdissect1}, our goal is to determine $A$ and $B$ to match \eqref{eq:rComp(Recall)} and \eqref{eq:WinqDemo} and have $A/3$ and $B/3$ be integral.

\begin{proof}[Proof of Theorem \ref{thm:QuintDiss(5mod12)}\,\ref{item:5(mod12)(1)}]
Suppose that $m \equiv -n \mod{3}$. Recalling that $3m\mu = 1 + ps$, as $p \equiv 2 \mod{3}$ it follows that $s \equiv 1 \mod{3}$. Then, let us further suppose\footnote{Compare with the relations $ms=2+3M$ and $ns=3N$ from section \ref{sec:1(mod12)}.}~that
    \begin{subequations}
    \begin{alignat}{3}
    \label{eq:m2(3)}
        m &\equiv 2 \mod{3} \qquad\text{and}\qquad &&ms &&= 2+3M, \\
    \label{eq:n1(3)}
        n &\equiv 1 \mod{3} \qquad\text{and}\qquad &&\hspace{0.33em}ns &&= 1+3N,
    \end{alignat}
and set
    \(
        (m+n)s = 3K,
    \)
    \end{subequations}
for integers $M$, $N$, and $K=M+N+1$.

With these assumptions and relation \eqref{eq:jtpShift}, the last term of \eqref{eq:rComp(Recall)} is
    \[
        \<-q^{\xa-Ky},-q^{\xb-nsp^{2}+msp^{2}};q^y\>_{\infty} 
        = q^{K(\xa-y)-\xs_K} \<-q^{\xa},-q^{\xb-nsp^2+msp^2};q^{y}\>_{\infty},
    \]
and we can rewrite the right-hand side of \eqref{eq:rComp(Recall)} as
    \(
    \label{eq:rComp}
    \begin{aligned}
        & q^{\xg}\<-q^{\xa},-q^{\xb};q^y\>_{\infty} 
        + q^{\xg+\xd+\xi+[K(\xa-y)-\xs_K]}\<-q^{\xa},-q^{\xb-nsp^{2}+msp^{2}};q^y\>_{\infty} \\
        &\qquad - q^{\xg+\xd}\<-q^{\xa-nsp^{2}},-q^{\xb+msp^{2}};q^y\>_{\infty} 
        - q^{\xg+\xi}\<-q^{\xa-msp^{2}},-q^{\xb-nsp^{2}};q^y\>_{\infty}.
    \end{aligned}
    \)
This latter expression is (somewhat) reminiscent of the right-hand side of \eqref{eq:WinqDemo}: 
in particular, we are motivated to set
    \[
        A := \xa \qquad\text{and}\qquad B := \xb - nsp^2
    \]
in \eqref{eq:rComp}, which changes said expression to
    \(
    \label{eq:rComp-2}
    \begin{aligned}
        & q^{\xg}\<-q^{A},-q^{B+nsp^2};q^y\>_{\infty} 
        + q^{\xg+\xd+\xi+[K(\xa-y)-\xs_K]} \<-q^{A},-q^{B+msp^2};q^{y}\>_{\infty} \\
        &\qquad -q^{\xg+\xd}\<-q^{A-nsp^2},-q^{B+Ky};q^{y}\>_{\infty} 
        - q^{\xg+\xi}\<-q^{A-msp^2},-q^{B};q^{y}\>_{\infty}.
    \end{aligned}
    \)
Now using \eqref{eq:m2(3)} and \eqref{eq:n1(3)}, we have
    \[
        q^{\xg}\<-q^{A},-q^{B+nsp^2};q^y\>_{\infty} 
            = q^{\xg+\lf[{-N(B+p^2)-\xs_{N}}\rh]} \<-q^{A},-q^{B+p^2};q^y\>_{\infty}
    \]
and
    \begin{multline}
        q^{\xg+\xd+\xi+\lf[K(\xa-y)\vphantom{{}^2}-\xs_{K}\rh]}\<-q^A,-q^{B+msp^2};q^y\>_{\infty} \\
        = q^{%
            \xg+\xd+\xi+\lf[K(\xa-y)\vphantom{{}^2}-\xs_{K}\rh]
            +\lf[{-M(B+2p^2)-\xs_{M}}\rh]%
            }%
            \<-q^A,-q^{B+2p^2};q^y\>_{\infty}.
        \notag
    \end{multline}
Comparing these with \eqref{eq:WinqDemo} and accounting for the $q^{\xc}$ in \eqref{eq:WinqDiss(m=-n(3))}, we initially ``define''
    \[
        \xc = \xc_{\xk} = \xg_{\xk}-N(B+p^2)-\xs_{N},
    \]
and are able to simplify this to the definition
    \(
    \label{eq:chi(k)}
        \xc_{\xk} = \xg_{k} + p(1-ns)\big({kn+\tfr{p+m-n(1+ps)}{6}}\big).
    \)
We note that formula \eqref{eq:chi(k)} makes it easy to see that $\xc_{\xk} \equiv \xg_{\xk} \equiv 3b\xk m \equiv r \mod{p}$.

Continuing with \eqref{eq:m2(3)}, \eqref{eq:n1(3)}, and \eqref{eq:chi(k)}, after some symbolic manipulations we find that the $r$-component \eqref{eq:rComp-2} is indeed
    \(
    \label{eq:rComp-3}
    \begin{aligned}
        &= q^{\xc}\<-q^A,-q^{B+p^2};q^y\>_{\infty} + q^{\xc+\fr{B}{3}}\<-q^{A},-q^{B+2p^2};q^y\>_{\infty} \\
            &\qquad - q^{\xc+\fr{A-B}{3}}\<-q^{A+p^2},-q^{B};q^y\>_{\infty} - q^{\xc+\fr{2A-B}{3}} \<-q^{A+2p^2},-q^B;q^y\>_{\infty}, \\
        &= q^{\xc}W(-q^{A/3},-q^{B/3},q^{p^2}),
    \end{aligned}
    \)
as desired.

If, in place of \eqref{eq:m2(3)} and \eqref{eq:n1(3)}, we assume that $m \equiv 1 \mod{3}$ and \mbox{$n \equiv 2\mod{3}$}, our derivations are nearly identical. Indeed, although our  initial ``definition'' of $\xc_{k}$ is different, we again obtain \eqref{eq:chi(k)} after some simplifications, and we again find \eqref{eq:rComp-2} and \eqref{eq:rComp-3} to be equal, completing the proof of part \ref{item:5(mod12)(1)} of the theorem.
\end{proof}

\begin{proof}[Proof of Theorem \ref{thm:QuintDiss(5mod12)}\,\ref{item:5(mod12)(2)}]
Now supposing that $m \equiv n \mod{3}$, let
    \begin{subequations}
    \begin{align}
    \label{eq:m1(m=n)}
        m &\equiv 1 \mod{3} \qquad\text{and}\qquad ms = 1+3M, \\
    \label{eq:n1(m=n)}
        n &\equiv 1 \mod{3} \qquad\text{and}\qquad ns = 1+3N,
    \end{align}
and let
    \[
        (m-n)s = 3K,
    \]
again for integers $M$, $N$, and $K=M-N$. 
    \end{subequations}
Let
    \[
        A := \xa_k - msp^2 \qquad\text{and}\qquad B:= \xb_k,
    \]
so that \eqref{eq:rComp(Recall)} becomes
    \(
    \label{eq:rComp(m=n)}
    \begin{aligned}
        & q^{\xg}\<-q^{A+msp^2},-q^{B};q^y\>_{\infty} 
        + q^{\xg+\xd+\xi}\<-q^{A-nsp^{2}},-q^{B+Ky};q^y\>_{\infty} \\ 
        &\qquad - q^{\xg+\xd}\<-q^{A+Ky},-q^{B+msp^{2}};q^y\>_{\infty} 
        - q^{\xg+\xi}\<-q^{A},-q^{B-nsp^{2}};q^y\>_{\infty}.
    \end{aligned}
    \)
This time, we find that
    \[
        q^{\xg}\<-q^{A+msp^2},-q^{B};q^y\>_{\infty} 
            = q^{\xg+\lf[{-M(A+p^2)-\xs_{M}}\rh]}\<-q^{A+p^2},-q^{B};q^y\>_{\infty}
    \]
and
    \begin{multline}
        q^{\xg+\xd+\xi}\<-q^{A-nsp^2},-q^{B+Ky};q^y\>_{\infty} \\
        = q^{%
            \xg+\xd+\xi+
            \lf[(N+1)(A+2p^2-y)-\xs_{N+1}\rh]+
            \lf[-KB-\xs_{K}\rh]%
        }
        \<-q^{A+2p^2},-q^{B};q^y\>_{\infty},
        \notag
    \end{multline}
and we are motivated to define 
    \[
        \xc_{\xk}^{*} = \xg_{\xk}-M(A+p^2)-\xs_{M}.
    \]
After expanding and simplifying, we find (cf.~\eqref{eq:chi(k)}) that
    \(
        \xc_{\xk}^{*} 
        = \xg_{\xk} + p(1-ms)\big({b+\xk m+\tfr{p-m(1+ps)-n}{6}}\big),
    \)
and ultimately that \eqref{eq:rComp(m=n)} is
    \begin{align*}
        &= q^{\xc_{\xk}^{*}}\<-q^{A+p^2},-q^{B};q^y\>_{\infty} + q^{\xc_{\xk}^{*}+\fr{A}{3}}\<-q^{A+2p^2},-q^{B};q^y\>_{\infty} \\
        &\qquad - q^{\xc_{\xk}^{*}+\fr{B-A}{3}}\<-q^{A},-q^{B+p^2};q^{y}\>_{\infty} 
        - q^{\xc_{\xk}^{*}+\fr{2B-A}{3}}\<-q^{A},-q^{B+2p^2};q^{y}\>_{\infty} \\
        &= q^{\xc_{\xk}^{*}}W(-q^{B/3},-q^{A/3};q^{p^2}), 
    \end{align*}
with $A = \xa_{\xk}-msp^2$ and $B = \xb_{\xk}$, noting the ordering of $A$ and $B$ here.

As in the case $m \equiv -n \mod{3}$, assuming that $m \equiv n \equiv 2 \mod{3}$ in place of \eqref{eq:m2(3)} and \eqref{eq:n1(3)} ultimately yields that \eqref{eq:rComp(m=n)} is equal to $q^{\xc_{\xk}^{*}}W(-q^{B/3},-q^{A/3},q^{p^2})$, and the proof of Theorem \ref{thm:QuintDiss(5mod12)} is complete.
\end{proof}

\begin{corollary}\label{Qdissectanvanc2}
    Let $p$, $m$, $n$ and $b$ be as in Theorem \ref{thm:QuintDiss(5mod12)}, and define $(a_t)_{t}$ via
        \(\label{p512cvaneq1}
            Q(q^{b m},q^p)Q(q^{b n},q^p)=\sum_{t=0}^{\infty} a_tq^t,
        \)
    with $a_{t} := 0$ for $t < 0$. In addition, let $m\bar{m} \equiv n\bar{n} \equiv 1 \mod{p}$, and let
        \[
            w \equiv \bar{2}(m+n) \pmod{p}.
        \]
    Then one has
        \[
            a_{pt+bw(1-3b\bar{m})} = a_{pt+bw(1-3b\bar{n})} = 0 \qquad (t \in \Z).
        \]
\end{corollary}

\begin{proof}
    First let $m \equiv -n \mod{3}$ and $0 \leq r < p$. The $r$-component of $Q(q^{bm},q^{p})Q(q^{bn},q^{p})$ is 
        \(
        \label{eq:rComp(5mod12)(ForZero)}
        \begin{aligned}
            & q^{\xc_{\xk}}
            W\!\big({-q^{\xa_{\xk}/3},-q^{(\xb_{\xk}-nsp^2)/3};q^{p^2}}\big) \\
            &\qquad 
            = \fr{q^{\xc_{\xk}} \<-q^{\xa_{\xk}/3},-q^{(\xb_{\xk}-nsp^2)/3};q^{p^2}\>_{\infty}}{(q^{p^2};q^{p^2})_{\infty}^{2}} 
            \times \big\<{q^{(\xa_{\xk}+\xb_{\xk}-nsp^2)/3},q^{(\xa_{\xk}-\xb_{\xk}+nsp^2)/3};q^{p^2}}\big\>_{\infty},
        \end{aligned}
        \)
    where $3b\xk m \equiv r \mod{p}$. By Lemma \ref{cor:<qx,qy>=0}, this component will be identically zero if either 
        \[
            \xa_{\xk}+\xb_{\xk} \equiv 0 \mod{p^2} \qquad\text{or}\qquad
            \xa_{\xk}-\xb_{\xk} \equiv 0 \mod{p^2}.
        \]
    Expanding these congruences using the definitions of $\xa_{\xk}$ and $\xb_{\xk}$, we quickly find that said congruences are equivalent to
        \(
        \label{eq:expRels(5mod12)(m=-n(3))}
            3b-n + 3\xk(m+n) \equiv 0 \mod{p} \qquad\text{and}\qquad
            3b-m + 3\xk(m-n) \equiv 0 \mod{p},
        \)
    respectively. Because $p = m^2 + n^2$, it is not possible that $m \pm n \equiv 0 \mod{p}$, so we may solve these congruences for $\xk \mod{p}$. In particular, using the relations
        \[
            m\bar{n} \equiv -\bar{m}n \mod{p}, \quad
            \overline{m+n} \equiv \bar{2}(\bar{m}+\bar{n}) \mod{p}, 
            \quad\text{and}\quad
            \overline{m-n} \equiv \bar{2}(\bar{m}-\bar{n}) \mod{p},
        \]
    we rearrange \eqref{eq:expRels(5mod12)(m=-n(3))} to find that
        \[
            3\xk \equiv \bar{2}n(\bar{m}+\bar{n})(1-3b\bar{n}) \mod{p} \qquad\text{and}\qquad
            3\xk \equiv \bar{2}m(\bar{m}-\bar{n})(1-3b\bar{m}) \mod{p},
        \]
    respectively, which are equivalent to 
        \[
            3b\xk m \equiv \bar{2}b(m+n)(1-3b\bar{n}) \mod{p} \quad\text{and}\quad
            3b\xk m \equiv \bar{2}b(m+n)(1-3b\bar{m}) \mod{p},
        \]
    respectively. Thus, one has $a_{pt+r}=0$ for all $t$ when 
        \[
            r \equiv bw(1-3b\bar{n}) \mod{p} \qquad\text{or}\qquad
            r \equiv bw(1-3b\bar{m}) \mod{p},
        \]
    as claimed. The proof when $m \equiv n \mod{3}$ is nearly identical.
\end{proof}

\begin{example}\label{p=17vanex}
Let $p=17=4^2+1^2$ and $b=2$ in Corollary \ref{Qdissectanvanc2}. Then
    \[
        w \equiv 11 \mod{17}, \quad
        bw(1-3b\bar{m}) \equiv 6 \mod{17}, \quad\text{and}\quad
        bw(1-3b\bar{n}) \equiv 9 \mod{17},
    \]
so that
    \[
        a_{17t+6} = a_{17t+9} = 0 \qquad \text{for all $t$.}
    \]
\end{example}

\begin{corollary}\label{p17mod2c}
Maintaining the assumptions of Corollary \ref{Qdissectanvanc2}, again let
    \[
        w \equiv \bar{2}(m+n) \mod{p}.
    \] 
Then for all integers $t$, the quantities $a_{pt+bw}$ and $a_{pt+b(w-3b)}$ are even; i.e., one has
    \(\label{p512c2vaneq2}
        a_{pt+bw} \equiv a_{pt+b(w-3b)} \equiv 0 \pmod{2} \qquad (t \in \Z).
    \)
\end{corollary}

\begin{proof}
The proof is similar to that of the previous corollary.
Letting \mbox{$m \equiv -n \mod{3}$} and fixing $0 \leq r \leq p-1$, the $r$-component of \eqref{eq:WinqDiss(m=-n(3))} has factors
    \(\label{p512c2vaneq3}
        \<-q^{\xa_{\xk}/3};q^{p^2}\>_{\infty} \quad\text{and}\quad \<-q^{(\xb_{\xk}-nsp^2)/3};q^{p^2}\>_{\infty}
    \)
from $W(-q^{\xa_{\xk}/3},-q^{(\xb_{\xk}-nsp^2)/3},q^{p^2})$, where again $3b\xk m \equiv r \mod{p}$.

If $\xa_{\xk} \equiv 0 \mod{p^2}$ or $\xb_{\xk} - nsp^2 \equiv 0 \mod{p^2}$, then the corresponding triple product in \eqref{p512c2vaneq3} has a factor $1-(-q^0) = 2$. It is easily seen that these are equivalent to 
    \[
        6b+6\xk m-(m+n) \equiv 0 \pmod{p} \qquad\text{and}\qquad 6\xk n + m-n \equiv 0 \mod{p},
    \]
respectively, which hold when
    \[
        3b\xk m \equiv b(w-3b) \mod{p} \qquad\text{and}\qquad 3b\xk m \equiv w \mod{p},
    \]
respectively. Thus, \eqref{p512c2vaneq2} holds when $m\equiv-n\mod{3}$; the proof when $m \equiv n \mod{3}$ follows a nearly identical argument.
\end{proof}

\begin{example}\label{p=17mod2ex}
Let $p = 17 = 4^2 + 1^2$ and $b=2$ in Corollary \ref{p17mod2c}. Then
    \[
        w \equiv 11 \mod{17}, \quad
        bw \equiv 5\mod{17}, \quad \text{and}\quad
        b(w-3b) \equiv 10\mod{17},
    \]
whereby
    \[
        a_{17t+5} \equiv a_{17t+10} \equiv 0 \pmod{2} \qquad (t \in \Z).
    \]
Indeed, this is reflected in the computations 
    \begin{align*}
    (a_{17t+5})_{t \geq 0} & =(\text{0, 0, 0, 0, $-2$, $-4$, $-8$, $-16$, $-28$, $-48$, $-82$,} \ldots), \\
    (a_{17t+10})_{t \geq 0} & = (\text{$2$, $4$, $10$, $20$, $40$, $72$, $130$, $220$, $368$, $594$, $948$,}\ldots).
    \end{align*}
\end{example}

\subsection{Sign patterns when \tops{$p\equiv 5\mod{12}$}{p≡5(mod 12)}}

{ % renewing \. spacing
\renewcommand{\.}{\mspace{0.9mu}}

\begin{corollary}\label{Qdissectanvanceqsign2}
    Let  $p$, $m$, $n$, $b$ and $s$ as in Theorem \ref{thm:QuintDiss(5mod12)},
    and define $(b_t)_{t}$ via
        \(\label{Qdissectc2eq1p5mod12}
            \frac{Q(q^{bm},q^{p})Q(q^{bn},q^p)}{(q^p;q^p)_{\infty}^2}=\sum_{t=-\infty}^{\infty}b_{t}q^{t}.
        \)
Fix $0 \leq r \leq p-1$ and let $0 \leq \xk \leq p-1$ be such that $3b\xk m\equiv r\pmod{p}$. 

If $m \equiv -n \mod{3}$ let
    \begin{align*}
        h &= \fr{\xa_{\xk}+\xb_{\xk}-nsp^2}{3p} = b + \xk(m+n) + p - \tf{1}{3}n(1+ps), \\
        h\.' &= \fr{\xa_{\xk}-\xb_{\xk}+nsp^2}{3p} = b + \xk(m-n) - \tf{1}{3}(m+n) + \tf{1}{3}n(1+ps),
    \end{align*}
and if $m \equiv n \mod{3}$ let
    \begin{align*}
        h &= \fr{\xb_{\xk} + \xa_{\xk} - msp^{2}}{3p} = b + \xk(m+n) + p + \tf{1}{3}(m-n) - \tf{1}{3}m(1+ps),\\
        h' &= \fr{\xb_{\xk} - \xa_{\xk} + msp^{2}}{3p} = -b - \xk(m-n) + \tf{1}{3}m(1+ps).
    \end{align*}
In addition, define $k$ and $\l$ so that
    \[
        h = kp+\l \qquad \text{with $0 \leq \l < p$},
    \]
and similarly define $k'$ and $\l\.'$. Then either $b_{pt+r} = 0$ for all $t$, or one has
    \[
        \sgn b_{pt+r} = (-1)^{k+k'} \qquad \text{for all $t > t_{0}(r)$}.
    \]
\end{corollary}

\begin{proof}
    Suppose that $m \equiv -n \mod{3}$. Then the $r$-component of $Q(q^{bm},q^{p})Q(q^{bn},q^{p})$ is
        \(
        \label{eq:rComp(5mod12)(ForSigns)}
        \begin{aligned}
            & q^{\xc_{\xk}}W(-q^{\xa_{\xk}/3},-q^{(\xb_{\xk}-nsp^2)/3},q^{p^2}) \\
            &\qquad 
            = \fr{q^{\xc_{\xk}} \<-q^{\xa_{\xk}/3},-q^{(\xb_{\xk}-nsp^2)/3};q^{p^2}\>_{\infty}}{(q^{p^2};q^{p^2})_{\infty}^{2}} \times \big\<{q^{(\xa_{\xk}+\xb_{\xk}-nsp^2)/3},q^{(\xa_{\xk}-\xb_{\xk}+nsp^2)/3};q^{p^2}}\big\>_{\infty},
        \end{aligned}
        \)
    where $3b\xk m \equiv r \mod{p}$. Letting
        \begin{align*}
            h &= \fr{\xa_{\xk}+\xb_{\xk}-nsp^2}{3p} = b + \xk(m+n) + p - \fr{n(1+ps)}{3}, \\
            h' &= \fr{\xa_{\xk}-\xb_{\xk}+nsp^2}{3p} = b + \xk(m-n) - \fr{m+n}{3} + \fr{n(1+ps)}{3},
        \end{align*}
    we have
        \(
        \label{eq:temp}
            \big\<{q^{(\xa_{\xk}+\xb_{\xk}-nsp^2)/3},q^{(\xa_{\xk}-\xb_{\xk}+nsp^2)/3};q^{p^2}}\big\>_{\infty} 
            = \<q^{ph},q^{ph'};q^{p^2}\>_{\infty}.  
        \)
    Defining $k$ and $\l$ as above, from \eqref{eq:jtpShift} we have
        \[
            \< q^{h};q^{p} \>_{\infty}
            = (-1)^{k} 
            q^{%
            -k\l-\fr{k(k-1)p}{2}
            }
            \< q^{\l};q^{p} \>_{\infty}.
        \]
    If $\l=0$ then \eqref{eq:temp} (and consequently \eqref{eq:rComp(5mod12)(ForSigns)}) is zero, as are all $b_{pt+r}$, and if $\l\neq 0$ then 
        \[
            \<q^{\l};q^{p}\>_{\infty}(q;q)^{-1}_{\infty} = \prod_{\substack{a=1\\[0.1em] a \not\equiv \pm\l \mod{p}}}^{p-1} (q^{a};q^{p})_{\infty}^{-1}.
        \]
    Writing a similar formula for $\<q^{\l\.'};q^{p}\>_{\infty}$, from this point the result follows from arguments like those in Lemma \ref{lem:PosCoeffs} and Corollary \ref{Qdissectanvanceqsign1}; the proof when $m \equiv n \mod{3}$ is similar.
\end{proof}

\begin{example}\label{p17signpat}
For $p=17=4^{2}+1^{2}$ and $b=2$, the signs of $b_{17t+r}$ for $0 \leq r < 17$ and $t > t_{0}(r)$ are given in the following table.
    \[
    \begin{array}{c|ccccccccccccccccc}
        r& 0 & 1 & 2 & 3 & 4 & 5 & 6 & 7 & 8 & 9 & 10 & 11 & 12 & 13 & 14 & 15 & 16 \\
        \hline
        \operatorname{sgn} b_{17t+r} & + & - & - & + & + & - & 0 & - & - & 0 & + & + & - & - & + & - & +  \\
    \end{array}
    \]
We note that the entries of ``$0$'' for $r=6$ and $r=9$ are consistent with Example \ref{p=17vanex}.

\end{example}

} % end renew \.

%%%%%%%%%%%%%%%%%%%%%%%%%%%%%%%%%%%%%%%%%%
\section{Combinatorial interpretations} 
%%%%%%%%%%%%%%%%%%%%%%%%%%%%%%%%%%%%%%%%%%
\label{sec:combinatorial}

Here we give two examples of how to interpret our dissection results combinatorially. For any $S \subset \N$ and any $n > 0$, let $D_S(n)$ denote the number of even-length, distinct-part partitions using only parts from $S$, minus the number of odd-length, distinct-part partitions using parts from $S$; in addition, set $D_S(0)=1$. 

\begin{example}\label{p13ex}
 If we set $p=13=2^{2}+3^{2}$ and $b=1$ in Theorem \ref{Qdissect1}, then after removing any negative exponents we get 
    \(
    \label{p-13partitionseq1}
    \begin{aligned}
    & Q(q^2,q^{13})Q(q^3,q^{13})
    = Q(q^{26},q^{169})^{2}
    - q^{27}Q(q^{52},q^{169})Q(q^{78},q^{169})
    - q^2Q(q^{39},q^{169})^{2} \\
    &\qquad\qquad%
    - q^{3}Q(q^{13},q^{169})Q(q^{39},q^{169}) 
    + q^{17}Q(q^{52},q^{169})Q(q^{65},q^{169}) \\
    &\qquad\qquad% 
    + q^5Q(q^{26},q^{169})Q(q^{52},q^{169})
    + q^{19}Q(1,q^{169})Q(q^{65},q^{169}) \\
    &\qquad\qquad%
    - q^7Q(q^{13},q^{169})Q(q^{52},q^{169})
    - q^{34}Q(q^{65},q^{169})Q(q^{78},q^{169}) \\
    &\qquad\qquad%
    + q^9Q(1,q^{169})Q(q^{13},q^{169})
    - q^{23}Q(q^{39},q^{169})Q(q^{78},q^{169}) \\
    &\qquad\qquad%
    - q^{24}Q(q^{13},q^{169})Q(q^{78},q^{169})
    + q^{12}Q(q^{26},q^{169})Q(q^{65},q^{169}).
    \end{aligned}
    \)
As has already been shown, the $r=6$ and $r=9$ components of the above are zero, due in both cases to their factors of $Q(1,q^{169})$. Here we give a combinatorial interpretation of these identically zero components, and we relate the nonzero components of the dissection to the corresponding pieces in the series on the left of \eqref{p-13partitionseq1} in a combinatorial way.
 
If both sides of \eqref{p-13partitionseq1} are divided by $(q^{13};q^{13})_{\infty}^2$, then the left side becomes 
    \(\label{p-13partitionseq2}
        \frac{ Q(q^2,q^{13})Q(q^3,q^{13}) }{ (q^{13};q^{13})_{\infty}^2 } =
        (q^2,q^3,q^{10},q^{11};q^{13})_{\infty}(q^7,q^9,q^{17},q^{19};q^{26})_{\infty} =
        \sum_{k=0}^{\infty} D_A(k)q^k,
    \)
where 
    \[
        A=\{m \in \N : \, m \equiv \pm 2,\pm 3,\pm 7,\pm 9,\pm 10,\pm 11 \mod{26} \}.
    \]
Thus we have that $D_A(13k+6)=D_A(13k+9)=0$ for all $t \geq 0$. As an example, for $k = 74 = 9+5(13)$ there are 158 even-length, distinct-part partitions of 74 using parts from $A$, and 158 odd-length, distinct-part partitions of 74 using parts from $A$.

On the other hand, if we consider the part of the dissection containing powers of $q$ with exponents congruent to $5 \pmod{13}$, then 
\[
    \sum_{k=0}^{\infty} D_A(13k+5)q^{13k+5} = \frac{q^5 Q(q^{26},q^{169}) Q(q^{52},q^{169})}{(q^{13};q^{13})_{\infty}^2},
\]
whereby
    \begin{align*}
        & \sum_{k=0}^{\infty} D_A(13k+5)q^{k} \\
        &\qquad = \frac{(q^2,q^{11},q^{13};q^{13})_{\infty} (q^9,q^{17};q^{26}){}_{\infty}}{(q;q)_{\infty}} \times 
        \frac{(q^4,q^9,q^{13};q^{13})_{\infty}  (q^5,q^{21};q^{26})_{\infty}}{(q;q)_{\infty}} \\
        &\qquad = \sum_{k=0}^{\infty}p_{13,5}(k);
    \end{align*}
here
    \[
        p_{13,5}(k) := \#\{(\pi_1,\pi_2):\,\pi_1 \in \mathcal{B}, \pi_2\in \mathcal{C}, |\pi_1|+|\pi_2|=k \},
    \]
where $\mathcal{B}$ and $\mathcal{C}$ are the sets of partitions with parts in $B$ and $C$, respectively,
    \begin{align*}
        B &=\{m \in \N :\, m \not\equiv 0,\pm 2,\pm 9,\pm 11,13 \pmod{26}, \\
        C &  =\{m \in \N :\, m \not\equiv 0,\pm 4,\pm 5,\pm 9,13 \pmod{26} \},
    \end{align*}
and $|\pi_i|$ denotes the sum of the parts in the partition $\pi_i$. Note that $\pi=\{\,\}$ is allowed, in which case $|\pi|=0$.

For example, consider $k=96=13(7)+5$. There are 609 partitions of 96 into an even  number of distinct parts from $A$, 547 partitions of 96 into an odd  number of distinct parts from $A$, so that $D_A(96)=609-547=62$, and there are also 62 bi-partitions/partition pairs $(\pi_1,\pi_2)$, with $\pi_2 \in \mathcal{B}$, $\pi_2\in \mathcal{C}$, and $|\pi_1|+|\pi_2|=7$.
\end{example}

\begin{example}\label{p17ex}
Setting $p = 17 = 4^2 + 1^1$ and $b=2$ in Theorem \ref{thm:QuintDiss(5mod12)}, then after removing any negative exponents we get 
    \begin{multline}
    \label{p-17partitionseq1}
        Q(q^2,q^{17})Q(q^8,q^{17}) = 
            W(-q^{136},-q^{68},q^{289}) 
                -q W(-q^{119},-q^{51},q^{289}) \\
            -q^2 W(-q^{119},-q^{34},q^{289})
                +q^3 W(-q^{119},-q^{68},q^{289})
                +q^{21} W(-q^{85},-q^{17},q^{289}) \\
            -q^{73} W(-q^{34},-1,q^{289})
                +q^{23} W(-q^{85},-q^{85},q^{289})
                -q^{24} W(-q^{136},-q^{119},q^{289}) \\
            -q^8 W(-q^{102},-q^{34},q^{289})
                -q^{43} W(-q^{51},-q^{51},q^{289})
                +q^{10} W(-q^{136},-1,q^{289}) \\
            +q^{28} W(-q^{68},-q^{51},q^{289})
                -q^{12} W(-q^{102},-q^{17},q^{289})
                -q^{13} W(-q^{136},-q^{102},q^{289}) \\
            +q^{14} W(-q^{102},-q^{85},q^{289})
                -q^{66} W(-q^{34},-q^{17},q^{289})
                +q^{33} W(-q^{68},-q^{17},q^{289}).
    \end{multline}
It was shown in Corollary \ref{p=17vanex} that the coefficients with indices $6$ and $9$ modulo 17 are zero, and this can be seen from \eqref{p-17partitionseq1} as well, since $W(x,x,q^k)=0$ for any $x$ and any $k>0$. Here the combinatorial interpretation of the vanishing coefficients is similar to that given in Example \ref{p13ex}, but the interpretation arising from equating the non-zero components of the dissection to the corresponding pieces in the series on the left of \eqref{p-17partitionseq1} is different.
 
If both sides of \eqref{p-17partitionseq1} are divided by $(q^{17};q^{17})_{\infty}^2$, then the left side becomes 
    \(\label{p-17partitionseq2}
        \fr{ Q(q^2,q^{17})Q(q^8,q^{17}) }{ (q^{17};q^{17})_{\infty}^2 } =
        (q^2,q^8,q^9,q^{15};q^{17})_{\infty} (q,q^{13},q^{21},q^{33};q^{34})_{\infty} = \sum_{k=0}^{\infty} D_A(k)q^k,
    \)
where 
    \[
        A = \{m \in \N :\, m \equiv \pm 1,\pm 2,\pm 8,\pm 9,\pm 13,\pm 15 \pmod{34} \}.
    \]
Thus we have that $D_A(17k+6)=D_A(17k+9)=0$ for all $k \geq 0$. For example for $k = 77 = 17(4)+9$, there are 56 partitions of 77 into an even  number of distinct parts from $A$, and  56 partitions of 77 into an odd  number of distinct parts from $A$.

On the other hand if we consider the part of the dissection containing powers of $q$ with exponent congruent to $2$ modulo 17, then 
    \[
        \sum_{k=0}^{\infty} D_A(17k+2)q^{17k+2} = \fr{-q^2 W(-q^{119},-q^{34},q^{289})}{(q^{17};q^{17})_{\infty}^2},
    \]
whereby
    \begin{align*}
    & \sum_{k=0}^{\infty} D_A(17k+2)q^{k} \\
    &\qquad = -(-q^2,-q^7,-q^{10},-q^{15};q^{17})_{\infty}
        \fr{(q^5,q^{12},q^{17};q^{17})_{\infty}}{(q;q)_{\infty}}
        \times \fr{(q^8,q^9,q^{17};q^{17})_{\infty}}{(q;q)_{\infty}} \\
    &\qquad = -\sum_{k=0}^{\infty} p_{17,2}(k);
    \end{align*}
here
    \[
        p_{17,2}(k) := \#\{(\pi_1,\pi_2,\pi_3):\,\pi_2 \in \mathcal{B}, \pi_2\in \mathcal{C}, \pi_3\in \mathcal{D},|\pi_1|+|\pi_2|+|\pi_3|=k \},
    \]
where $\mathcal{B}$ is the set of partitions with distinct parts in $B$, $\mathcal{C}$ and $\mathcal{D}$ are the sets of partitions with parts in $C$ and $D$, respectively, and 
    \begin{align*}
        B &= \{m \in \N :\, m \equiv \pm 2,\pm 7 \pmod{17} \}, \\
        C &= \{m \in \N :\, m \not\equiv 0,\pm 5 \pmod{17} \}, \\
        D &= \{m \in \N :\, m \not\equiv 0,\pm 8 \pmod{17} \}.
    \end{align*}
Again we again allow that $\pi=\{\,\}$ with $|\{\,\}|=0$.

For example, consider $k=189=17(11)+2$. There are 5013 partitions of 189 into an even  number of distinct parts from $A$, 5989 partitions of 189 into an odd  number of distinct parts from $A$, so that $D_A(189)=5013-5989=-976$, and there are also 976 partition triples $(\pi_1,\pi_2, \pi_3)$, with $\pi_1 \in \mathcal{B}$, $\pi_2\in \mathcal{C}$, $\pi_3\in \mathcal{D}$ and $|\pi_1|+|\pi_2|+|\pi_3|=11$.
\end{example}

%%%%%%%%%%%%%%%%%%%%%%%%%%%%%%%%%%%%%%%%%%
\section{Concluding remarks} \label{sec:conclusions}
%%%%%%%%%%%%%%%%%%%%%%%%%%%%%%%%%%%%%%%%%%

This paper was motivated by the experimental discovery of vanishing coefficients in products $Q(q^r,q^p)Q(q^s,Q^p)$ with $r, s > 0$ and $p$ a prime equivalent to 1 modulo 4. Further experiment suggests that some products $Q(q^r,q^p)Q(q^s,Q^p)Q(q^t,q^p)$ also exhibit vanishing coefficients in progressions modulo $p$; akin to the example in the introduction, for the sequence $(a_k)_{k \in \Z}$ defined via
    \[
        Q(q,q^{13})Q(q^3,q^{13})Q(q^4,q^{13}) = \sum_{k=0}^{\infty} a_k q^k,
    \]
one has
    \[
        a_{13k+2} = a_{13k+4} = a_{13k+10} = 0 \qquad
        \text{for all $k$}.
    \]
Furthermore, it appears that in the case of products $Q(q^i,q^{p})Q(q^j,q^{p})Q(q^k,q^{p})$, this phenomenon of vanishing coefficients is not restricted only those $p \equiv 1 \pmod{4}$; indeed, experiment also suggests that if 
    \[
        Q(q^2,q^{19})Q(q^3,q^{19})Q(q^5,q^{19}) = \sum_{k=0}^{\infty} b_k q^k,
    \]
then 
    \[
        b_{19k+4}=b_{19k+5}=b_{19k+16} = 0 \qquad \text{for all $k$}.
    \]
We leave it to the reader to spot a possibly required condition, suggested by the two examples, on the $i$, $j$ and $k$ in $Q(q^i,q^{p})Q(q^j,q^{p})Q(q^k,q^{p})$ for vanishing to occur.

%%%%%%%%%%%%%%%%%%%%%%%%%%%%%%%%%%%%%%%%%%
%%%%%%%%%%    BIBLIOGRAPHY    %%%%%%%%%%%%
%%%%%%%%%%%%%%%%%%%%%%%%%%%%%%%%%%%%%%%%%%

{\allowdisplaybreaks

}

%%%%%%%%%%%%%%%%%%%%%%%%%%%%%%%%%%%%%%%%%%
%%%%%%%%%%   DOCUMENT END   %%%%%%%%%%%%%%
%%%%%%%%%%%%%%%%%%%%%%%%%%%%%%%%%%%%%%%%%%
\end{document}